\documentclass[11pt]{article}

\usepackage{amsmath,amssymb,amsthm,mathtools}
\usepackage[margin=1in]{geometry}
\usepackage{enumitem}
\usepackage{graphicx}
\usepackage[colorlinks=true,allcolors=blue]{hyperref}

\numberwithin{equation}{section}

\theoremstyle{plain}
\newtheorem{theorem}{Theorem}[section]
\newtheorem{proposition}{Proposition}[section]
\newtheorem{corollary}{Corollary}[section]

\theoremstyle{definition}
\newtheorem{definition}{Definition}[section]
\newtheorem{example}{Example}[section]
\newtheorem{problem}{Problem}[section]

\theoremstyle{remark}
\newtheorem{remark}{Remark}[section]

\newcommand{\PP}{\mathbb P}
\newcommand{\OO}{\mathcal O}

\DeclareMathOperator{\ord}{ord}
\DeclareMathOperator{\Supp}{Supp}

\begin{document}

\title{Degree growth, orbit graphs, and functoriality for birational dynamical systems}
\author{ Tomoyuki Takenawa\\[3pt] 
\small Faculty of Marine Technology,\\ 
\small Tokyo University of Marine Science and Technology\thanks{ 2-1-6 Etchujima, Koto-ku, Tokyo 135-8533, Japan. }\\[2pt] 
\small \texttt{takenawa@kaiyodai.ac.jp} }
\date{}
\maketitle

\begin{abstract}
The purpose of this paper is to give a natural divisor-theoretic formulation
of the counting method introduced by Halburd for computing degree growth,
in a form applicable to birational dynamical systems on varieties of
arbitrary dimension. Instead of counting only preimages of special values, we follow
time-indexed divisorial conditions through singularity patterns.  These
conditions are recorded on normalized finite-window orbit graphs, where
the relevant multiplicities are realized as divisorial valuations of
pullbacks of time-indexed divisors.  This construction explains how the
elementary computations appearing in singularity patterns can be interpreted
as degree relations on a single normal variety.  
For divisors with globally generated associated line bundles, we then
show that further relations arise from the failure of functoriality of
pullbacks: when the center of a prime divisor enters the corresponding
divisor-dependent indeterminacy locus, an effective degree-drop divisor
appears.  Under
suitable finite-type assumptions, the two kinds of relations lead to
closed linear difference systems governing degree sequences.  Several
examples, including higher-dimensional ones, demonstrate that the two
mechanisms are complementary and that their combination determines the
degree growth in cases where either mechanism alone is insufficient.
\end{abstract}


\section{Introduction}

Degree growth of iterates is a basic invariant in the study of birational dynamical systems. In this paper we consider birational self-maps of smooth projective varieties and study degree sequences associated with their iterates. The exponential growth rate of such a sequence measures the complexity of the dynamics, and the logarithm of this growth rate is the algebraic entropy in the sense of Bellon--Viallet \cite{BellonViallet}. More broadly, the iteration of rational and meromorphic maps is a central theme in higher-dimensional complex dynamics, going back to early works of Forn\ae ss--Sibony and others \cite{FornaessSibony1994}.

This point of view is closely related to algebraic stability. On surfaces, Diller--Favre showed that bimeromorphic surface maps can be made algebraically stable after a suitable modification, so that degree growth is governed by the induced linear action on cohomology \cite{DillerFavre}.

A related geometric picture appears in the Okamoto--Sakai theory of
Painlev\'e equations and their discrete analogues.  In this theory, spaces
of initial conditions are constructed by resolving singularities, and the
resulting families of rational surfaces carry affine root-system
symmetries \cite{Okamoto1979,Sakai}.  Discrete Painlev\'e systems are
typically non-autonomous: their evolution is generally an isomorphism
between spaces of initial conditions at different parameter values, rather
than a self-map of a single surface.  In this setting, the degree growth with 
respect to the discrete time \(n\) is at most quadratic, 
and hence the algebraic entropy is zero.  In the present paper, however, we restrict our
attention to autonomous birational dynamical systems.  Accordingly, when
we refer below to maps related to Painlev\'e theory, we mean autonomous or
suitably specialized cases in which the evolution is viewed as a
birational self-map.

Closely related ideas have also been applied to two-dimensional dynamical
systems beyond the Painlev\'e setting, including the Hietarinta--Viallet
mapping.  In contrast to discrete Painlev\'e systems, the Hietarinta--Viallet
mapping has exponential degree growth and hence positive algebraic
entropy.  After resolving the relevant singularities, the induced linear
maps on, or between, the corresponding Picard groups govern the degree
sequence and hence the algebraic entropy \cite{Takenawa2001b}.

In higher dimensions, however, an analogous geometric understanding is not
yet sufficiently developed.  For higher-dimensional Cremona-type maps,
methods for computing degree growth by regularization have been studied
\cite{BedfordKim2004}, and related linear fractional recurrences have also
been analyzed using orbits of exceptional divisors and actions on Picard
groups after blow-ups \cite{BedfordKim2006,BedfordKim2011}.  In general,
the behavior of indeterminacy loci and exceptional divisors becomes
significantly more complicated in higher dimensions, and it is not easy to
explicitly construct all blow-ups and compute the action on the Picard
group or the N\'eron--Severi group.  In some well-structured
higher-dimensional autonomous examples, including examples related to
discrete Painlev\'e theory, one can construct a suitable birational model
on which the map becomes a pseudo-automorphism, and compute degree growth
from the induced action on divisor classes
\cite{Takenawa2004P3,CarsteaTakenawa2019,Takenawa2021,
StokesTakenawaCarstea2025}.  Such an analysis, however, depends on the
existence of a good birational model and cannot be directly applied to
general higher-dimensional birational dynamical systems.

Another important background comes from singularity confinement.
Singularity confinement, introduced by Grammaticos--Ramani and others, is
a powerful tool in the study of discrete integrable systems
\cite{GrammaticosRamani}.  It examines whether the degrees of freedom
temporarily lost under iteration of a map are recovered after finitely many
steps.  It was later shown by the example of Hietarinta--Viallet that
singularity confinement alone does not sufficiently characterize
integrability, and the viewpoints of degree growth and algebraic entropy
became important \cite{HietarintaViallet,BellonViallet}.  In this context,
Halburd gave an efficient method for computing degree sequences by
counting preimages of values appearing in singularity patterns
\cite{Halburd2017}.  We call this method Halburd's counting method in this paper.
Ramani, Grammaticos, Willox, Mase and collaborators further developed
Halburd-type counting methods for deriving degree growth from singularity
patterns
\cite{RamaniGrammaticosWilloxMase2017,MaseWilloxRamaniGrammaticos2019,
WilloxMaseRamaniGrammaticos2024,RamaniGrammaticosCarsteaWillox2025,
GrammaticosRamaniCarsteaWillox2025}.
More recently, Mase gave a rigorous justification of
Halburd's method for a class of lattice equations
\cite{Mase2026Lattice}.
The proof exploits the infinitely many independent initial
variables of a lattice initial-value problem.
Under suitable conditions, each movable singularity pattern
has a unique starting point and coincides, up to translation
and multiplicity, with a basic singularity pattern.
Consequently, the degree relations obtained by counting
singularities are exact.

A further related method is that of Kanki--Mada--Tokihiro and
collaborators, which directly controls the Laurent property,
irreducibility, and coprimeness of rational functions appearing under
iteration \cite{KankiMadaTokihiro,KankiMadaMaseTokihiro}.  More recently,
Alonso--Suris--Wei gave a method for determining degree growth using
factorizations of pulled-back polynomials and the propagation of indices
associated with blow-ups \cite{AlonsoSurisWeiI,AlonsoSurisWeiII}.

The starting point of this paper is to extend the viewpoint behind
Halburd's counting method: instead of counting only preimages of special
values, we follow time-indexed divisorial conditions through singularity
patterns.  However, these elementary computations are not, by themselves,
global divisor equalities on the original phase spaces.  A condition that
appears as a divisor at one time may have a center of higher codimension
at another time, and the divisorial components observed in a singularity
pattern do not live on a single fixed variety in a functorial way.  To
treat these computations geometrically, one needs an additional device.
In this paper, this device is the normalized finite-window orbit graph.
The prime Weil divisors on this graph record the divisorial components
traced by singularity patterns, while their centers under the time
projections describe how these components appear at each time.

This construction leads to two kinds of degree relations.  The first kind
comes from the decomposition of pullbacks of time-indexed divisors on the
normalized orbit graph.  The multiplicities appearing in this
decomposition are divisorial valuations, and applying degrees to the
decomposition produces relations among the degrees of the iterates.  We
call these Halburd-type divisor decompositions and the resulting degree
relations Halburd-type degree relations.  They include the usual
Halburd-type counting relations, but are not limited to the counting of
preimages of special values.

The second kind comes from a different mechanism.  It concerns degree
drops caused by common factors in iterated rational representations.  For divisors with globally generated associated line bundles, we
describe such cancellations by effective degree-drop divisors arising
from the failure of functoriality of pullbacks; their prime
decompositions give degree-drop relations.  We call the method combining
these two kinds of relations the divisor-orbit decomposition method.

A technical ingredient used in this paper is a criterion for functoriality
of pullbacks for rational maps on smooth projective varieties.  More
precisely, for smooth projective varieties \(X,Y,Z\), dominant rational
maps \(f:X\dashrightarrow Y\), \(g:Y\dashrightarrow Z\), and a very ample
divisor \(A\) on \(Z\), we give a necessary and sufficient condition for
the equality \(f^*g^*A=(g\circ f)^*A\) to hold.  
The condition is that no prime divisor
on \(X\) is mapped by \(f\) into the indeterminacy locus of \(g\).
Criteria relating functoriality of pullbacks to exceptional or divisorial components entering indeterminacy loci are central in Diller--Favre's algebraic stability for surface maps and in higher-dimensional work of Bedford--Kim, Roeder, and Bayraktar \cite{DillerFavre,BedfordKim2004,BedfordKim2006,BedfordKim2011, Roeder2014,Bayraktar2013}.
The feature of our formulation is that the obstruction to
functoriality is kept as an actual effective divisor, whose prime
components and multiplicities are described explicitly.  For this reason
the equivalence needed for degree drops appears naturally: the absence of
such divisorial components gives functoriality, while the presence of one
gives a nonzero component of the degree-drop divisor and hence forces the
pullback equality to fail.

Although this criterion for very ample divisors gives a necessary and
sufficient condition for functoriality, the degree-drop relations used
later require a more explicit form of the same phenomenon.  Namely, for a
chosen effective Cartier divisor \(D\) with \(\OO_Z(D)\) globally
generated, we need to describe how the functoriality relation
\[
f^*g^*D=(g\circ f)^*D
\]
fails, rather than only whether it holds.  We therefore extend the same
common-factor argument from very ample divisors to base-point-free
effective Cartier divisors, and obtain the effective degree-drop divisor
\[
f^*g^*D-(g\circ f)^*D.
\]
Its prime components and multiplicities give the divisor-theoretic data
that later appear in the degree-drop relations.

We next introduce finite-window orbit graph varieties.  For a finite
interval \(I=[a,b]\subset\mathbb Z\), we define
\begin{equation}
\label{eq:intro-orbit-graph}
\Omega_I\subset \prod_{i=a}^b X_i
\end{equation}
as the Zariski closure of orbit segments for which all successive images
are defined.  After normalization, prime Weil divisors on the orbit graph
represent the divisorial data used in the degree relations.  A singularity
pattern is then expressed by recording the centers of such a divisor under
the time projections.  In this way, components that appear at a given time
only as points or lower-dimensional subvarieties are not lost: they are
treated as divisors on the normalized orbit graph.  The remaining time
coordinates retain the parameters along these components and thus play
the role of exceptional parameters in the multiplicity computations.

The advantage of this formulation is that it extracts only the divisorial
data relevant to degree growth, without requiring a full construction of
spaces of initial conditions or direct control of all iterated rational
representations.  It is not intended to recover the full symmetry
structure of an integrable system.  Rather, it provides a finite-dimensional
description when the relevant divisorial components close into a finite
set.  If they do not, one must instead use infinitely many auxiliary
variables or a generating-function-type description.  This is consistent
with general frameworks for infinite-dimensional divisor data, such as
b-divisors and Picard--Manin-type spaces \cite{DangFavre}, and with known
examples of transcendental dynamical degrees \cite{BellDillerJonsson}.

The organization of the paper is as follows.  Section 2 presents two basic
examples, illustrating a degree-drop relation and a Halburd-type divisor
decomposition.  Section 3 studies the failure of functoriality of
pullbacks, first calculates degree drops directly in homogeneous
coordinates for projective and multiprojective spaces, and then
describes effective degree-drop divisors for general base-point-free
linear systems.
Section 4 introduces normalized finite-window orbit graphs and
singularity patterns as center sequences of divisorial valuations.
Section 5 formulates the divisor-orbit decomposition method as a system
of linear degree relations.  Sections 6 and 7 apply the method to a
conic-bundle example and a four-dimensional multiplicative map,
respectively.  Section 8 treats a Riccati-type linearizable map and
illustrates the need for infinitely many auxiliary variables when the
relevant divisor sequences do not close into a finite set.

Throughout the paper, all varieties are defined over \(\mathbb C\).

\section{Basic examples}
In the examples, equalities of degrees (or degree vectors) are measured
in the initial variables; components of degree zero on the normalized orbit
graph may be omitted. These equalities do not assert that all prime divisors
of the graph have been listed. The complete
orbit-graph decomposition is formulated in Section~4.

We introduce the notation for degrees used in this paper only to the extent needed for the examples. The general definition of divisor degrees will be given in later sections.

First consider the case \(X=\PP^N\). Let
\[
f:\PP^N\dashrightarrow\PP^N
\]
be a birational self-map. We write its \(n\)-th iterate, using initial homogeneous coordinates
\[
\mathbf{x}_0=[x_{0,0}:x_{1,0}:\cdots:x_{N,0}],
\]
as
\[
f^n(\mathbf{x}_0)
=
[P_{0,n}(\mathbf{x}_0):\cdots:P_{N,n}(\mathbf{x}_0)].
\]
Here \(P_{0,n},\ldots,P_{N,n}\) are homogeneous polynomials with no common factor. We define
\[
d_n:=\deg f^n:=\deg P_{0,n}.
\]
Equivalently, \(\deg P_{i,n}\) is independent of \(i\).

Next consider the case \(X=(\PP^1)^M\). Let
\[
f:(\PP^1)^M\dashrightarrow(\PP^1)^M
\]
be a birational self-map. We write its \(n\)-th iterate in affine coordinates, using initial affine coordinates
\[
\mathbf{x}_0=(x_{1,0},\ldots,x_{M,0}),
\]
as
\[
f^n(\mathbf{x}_0)
=
(x_{1,n},\ldots,x_{M,n})
=
\left(
\frac{p_{1,n}(\mathbf{x}_0)}{q_{1,n}(\mathbf{x}_0)},
\ldots,
\frac{p_{M,n}(\mathbf{x}_0)}{q_{M,n}(\mathbf{x}_0)}
\right),
\]
where \(p_{i,n},q_{i,n}\) have no common factor. We set
\[
\deg_{x_{j,0}}x_{i,n}
:=
\max\{\deg_{x_{j,0}}p_{i,n},\deg_{x_{j,0}}q_{i,n}\}.
\]
We also define the degree of \(f^n\) by
\[
d_n:=\deg f^n
:=
\max_{1\leq i,j\leq M}\deg_{x_{j,0}}x_{i,n}.
\]

\subsection[A basic example on P3]{A basic example on \(\PP^3\)}
\label{example:working-cremona-degree-drop}
Consider the standard Cremona transformation on \(\PP^3\),
\[
\sigma[x_0:x_1:x_2:x_3]
=
[x_1x_2x_3:x_0x_2x_3:x_0x_1x_3:x_0x_1x_2],
\]
and the linear automorphism \(A\in\operatorname{PGL}_4\)
\[
A=
\begin{pmatrix}
0&1&1&1\\
0&1&1&0\\
1&1&0&1\\
1&0&1&1
\end{pmatrix}.
\]
We study the birational dynamical system on \(\PP^3\) defined by \(f=A\circ\sigma\).

Then
\[
f[x_0:x_1:x_2:x_3]=[y_0:y_1:y_2:y_3],
\]
where
\[
\begin{aligned}
y_0&=x_0(x_1x_2+x_1x_3+x_2x_3),\\
y_1&=x_0x_3(x_1+x_2),\\
y_2&=x_2(x_0x_1+x_0x_3+x_1x_3),\\
y_3&=x_1(x_0x_2+x_0x_3+x_2x_3).
\end{aligned}
\]
Applying \(\sigma\) once more gives
\[
\sigma[y_0:y_1:y_2:y_3]
=
[y_1y_2y_3:y_0y_2y_3:y_0y_1y_3:y_0y_1y_2].
\]
A direct computation shows that these four homogeneous forms have the common factor \(x_0\) with multiplicity \(1\), and hence
\[
\deg(\sigma\circ f)=8<9=\deg\sigma\cdot\deg f.
\]
Since \(A\) is only a linear isomorphism and does not change degrees,
\[\deg f^2=8<9=(\deg f)^2.\]

Let us view this degree drop geometrically. The indeterminacy locus of \(\sigma\) is the union of the six coordinate lines
\[
I(\sigma)=\bigcup_{0\le i<j\le3}\{x_i=x_j=0\}.
\]
On the other hand, the coordinate hyperplane
\[
H_0=\{x_0=0\}
\]
is contracted by \(f=A\circ\sigma\) to
\[
f(H_0)=A[1:0:0:0]=[0:0:1:1].
\]
This point satisfies
\[
[0:0:1:1]\in I(\sigma).
\]
Thus \(H_0\) is mapped by \(f\) into the indeterminacy locus of the next map \(\sigma\), and consequently the homogeneous representation of the composite \(\sigma\circ f\) has the common factor \(x_0\). Moreover, the image of the generic point of \(H_0\) under \(\sigma\circ f\), as determined by the representation after cancellation, is the line \(x_2=x_3=0\) (Figure~\ref{fig:cremona-example}).

\begin{figure}[htbp]
\centering
\includegraphics[width=.4\textwidth]{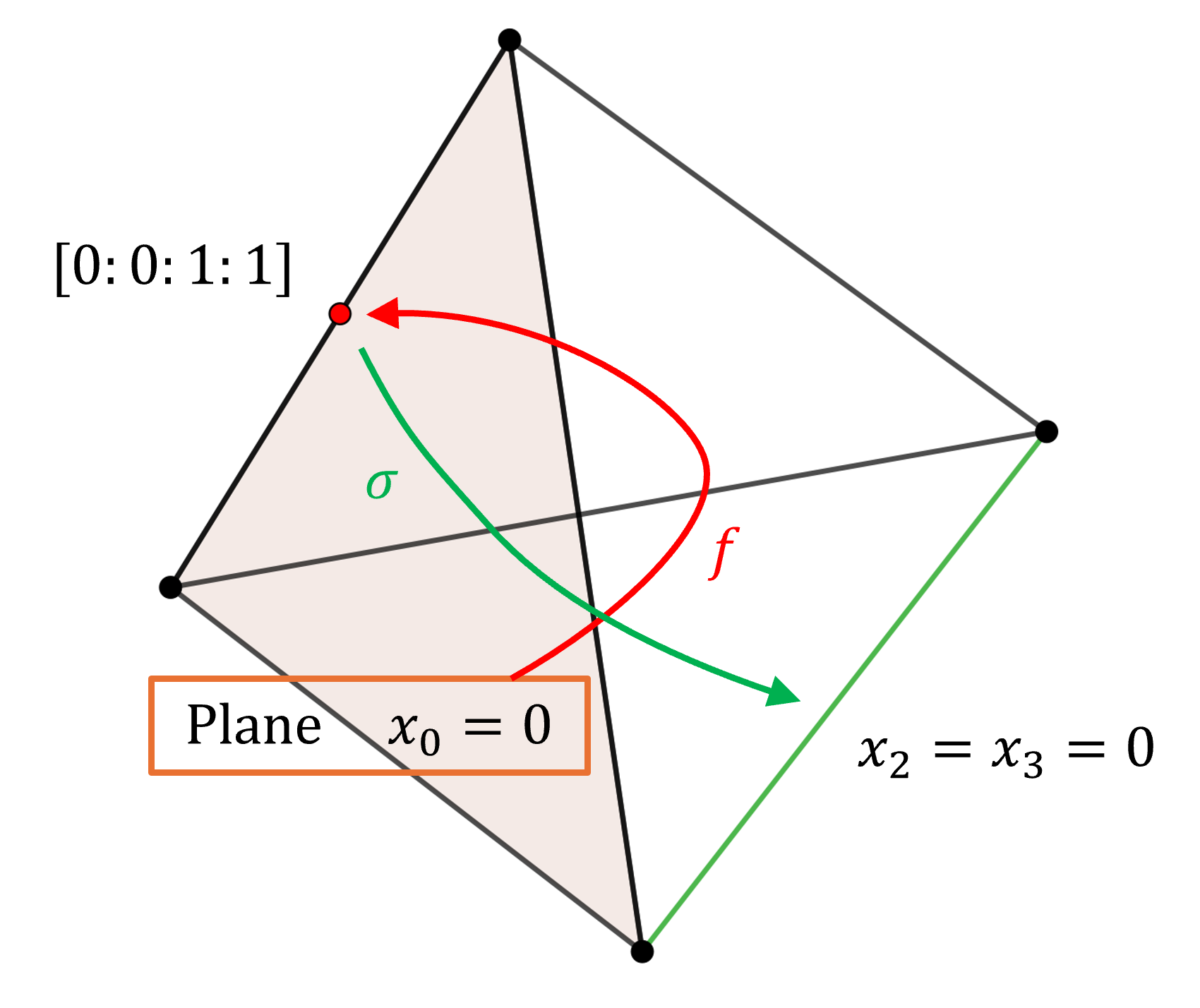}
\caption{A schematic picture of the behavior of the hyperplane \(H_0=(x_0=0)\) under \(f\) and \(\sigma\). The hyperplane \(H_0\) is contracted by \(f\) to a point, and this point belongs to the indeterminacy locus of \(\sigma\).}
\label{fig:cremona-example}
\end{figure}

The orbit of \(H_0\) is
\[
x_0=0 \to [0:0:1:1] \to
\text{the line } \{x_1=x_0,\ x_2=x_0+x_3\}
\to \text{a conic} \to \text{a quartic} \to \cdots .
\]
That this orbit is never subsequently contained in the plane \(x_0=0\) can be checked, for example, by computing the orbit of the point \([2:2:3:1]\), which lies on the line above.  Since all its
coordinates are positive and the coordinate polynomials of $f$
have positive coefficients, all forward iterates of this point have
positive coordinates. The other divisors contracted by \(f\) are \(x_i=0\), \(i=1,2,3\). The orbits of their generic points are
\begin{align*}
&x_1=0 \to [1:1:1:0] \to [1:0:1:1] \to [1:1:1:0] \to \cdots,\\
&x_2=0 \to [1:1:0:1] \to [1:1:0:1] \to [1:1:0:1] \to \cdots,\\
&x_3=0 \to [1:0:1:1] \to [1:1:1:0] \to [1:0:1:1] \to \cdots,
\end{align*}
and these orbits are also not contained in \(I(\sigma)\).

Using this singularity pattern, the components contributing to the degree of the divisor \((x_{0,n}=0)\) are
\begin{equation}
\label{eq:p3-basic-halburd-decomposition}
\deg(x_{0,n}=0)
=\deg((x_{0,n}=0)^\circ)+\deg((x_{0,n-1}=0)^\circ).
\end{equation}
Here \((x_{0,n}=0)^\circ\) corresponds to generic points of the divisor \((x_{0,n}=0)\); its precise definition will be given in later sections. If we denote the degree of \((x_{0,n}=0)^\circ\) by \(t_n\), then the degree equality above gives
\begin{equation}
\label{eq:cremona-degree-decomposition}
d_n=t_n+t_{n-1}.
\end{equation}

In computing the degree \(d_{n+1}\) of \(f^{n+1}\), since \(f\) is given by cubic forms and the only factor that disappears by cancellation is \((x_{0,n-1}=0)^\circ\), we get
\begin{equation}
\label{eq:cremona-degree-drop}
d_{n+1}=3d_n-t_{n-1}.
\end{equation}

Thus we obtain the recurrence
\begin{equation}
\label{eq:cremona-closed-recurrence}
d_{n+2}=2d_{n+1}+2d_n.
\end{equation}
Since the degree is measured with respect to $\mathbf{x}_0$,
the initial values are $d_0=1$ and $d_1=3$, and hence
\[
d_2=8,\quad d_3=22,\quad d_4=60,\ldots,
\]
in agreement with direct computation. The general term is
\[
d_n=
\frac{
(2+\sqrt{3})(1+\sqrt{3})^n
+(\sqrt{3}-2)(1-\sqrt{3})^n
}{2\sqrt{3}}.
\]
Consequently, the first dynamical degree and
the algebraic entropy are
\[
\lambda_1(f)=1+\sqrt{3},
\qquad
h_{\mathrm{alg}}(f)=\log(1+\sqrt{3}).
\]

The auxiliary sequence satisfies $t_0=1$ and $t_1=2$.
The latter follows from the factorization
\[
x_{0,1}=x_{0,0}
(x_{1,0}x_{2,0}+x_{1,0}x_{3,0}+x_{2,0}x_{3,0}).
\]

\subsection[A basic example on P1 x P1]{A basic example on \(\PP^1\times \PP^1\)}
\label{ex:dp2}

Consider the birational dynamical system on \(X=\PP^1\times\PP^1\) given by
\begin{equation}
\label{eq:p1p1-map}
(x_{n+1},y_{n+1})
=
\varphi(x_n,y_n)
=
\left(1-\frac1{x_n}-y_n,\ x_n\right).
\end{equation}
This map is integrable in the sense that, after resolving the relevant
base points, it preserves an elliptic fibration; the associated rational
elliptic surface has a reducible singular fiber of type \(D_5^{(1)}\).
It is also studied in \cite{RamaniGrammaticosWilloxMase2017}.  The map is
equivalent to the second-order recurrence
\(x_{n+1}=1-1/x_n-x_{n-1}\), but in this paper we always treat it as an
orbit of \((x_n,y_n)\in\PP^1\times\PP^1\).

Let the initial value be \((x_0,y_0)\). In this example one checks directly that the degree sequence is given by
\[
d_n:=\deg_{x_0}x_n.
\]
The first few values are
\[
d_0=1,\qquad d_1=1,\qquad d_2=2,\qquad
d_3=3,\qquad d_4=5,\qquad d_5=6.
\]

We first list the singularity patterns of this map. They are obtained by tracking generic points of divisors on which the Jacobian vanishes. The main singularity pattern involved in the degree calculation starts from the generic point of \(x=0\). Indeed, putting
\[
x=\varepsilon,
\qquad y=\alpha,
\]
and following the principal parts of the Laurent expansions, we obtain
\[
(x=0)^\circ
\to
(\infty,0)
\to
(1,\infty)
\to
(\infty,1)
\to
(0,\infty)
\to
(y=0)^\circ.
\]
Thus, including the generic components before and after it, we obtain the confined singularity pattern
\[
\text{curve}^{\circ}
\to
(x=0)^{\circ}
\to
(\infty,0)
\to
(1,\infty)
\to
(\infty,1)
\to
(0,\infty)
\to
(y=0)^{\circ}
\to
\text{curve}^{\circ}.
\]
Here \({}^{\circ}\) denotes the component obtained by taking the generic point of the subvariety specified by the indicated condition and tracking its iterates. The precise definition will be given in later sections.

\begin{figure}[htbp]
\centering
\includegraphics[width=.4\textwidth]{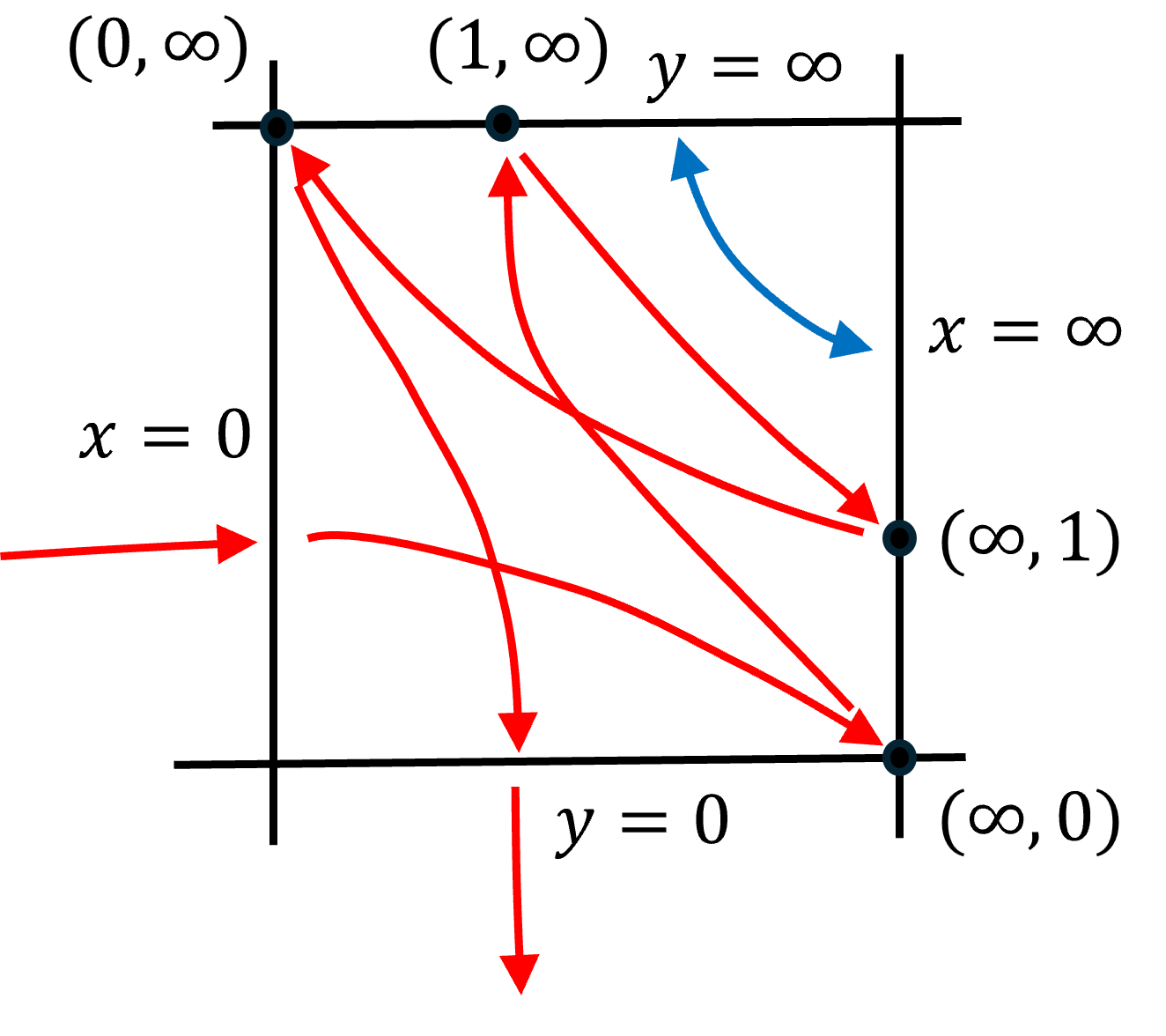}
\caption{A singularity pattern of the dynamical system on \(\PP^1\times\PP^1\). A contracted divisor reappears as a divisor after five time steps.}
\label{fig:p1p1-example}
\end{figure}

From this singularity pattern, the degree-contributing components of \((x_n=0)\) at time \(n\) are as follows. Since \((0,\infty)_n\) corresponds, by the singularity pattern, to \((x_{n-4}=0)^\circ\) four time steps earlier, we can write
\begin{equation}
\label{eq:p1p1-zero-decomposition}
\deg_{x_0}(x_n=0)
=\deg_{x_0}((x_n=0)^\circ)
+\deg_{x_0}((x_{n-4}=0)^\circ).
\end{equation}
If \(t_n\) denotes the degree, with respect to \(x_0\), of the divisor corresponding to \((x_n=0)^\circ\), then this degree equality gives
\begin{equation}
\label{eq:p1p1-zero-degree}
d_n=t_n+t_{n-4}.
\end{equation}

Next consider the divisor \((x_n=\infty)\) at time \(n\). By the same singularity pattern, \((\infty,0)_n=(x_{n-1}=0)^\circ\) and \((\infty,1)_n=(x_{n-3}=0)^\circ\). Therefore
\begin{equation}
\label{eq:p1p1-infty-decomposition}
\deg_{x_0}(x_n=\infty)
=\deg_{x_0}((x_n=\infty)^\circ)
+\deg_{x_0}((x_{n-1}=0)^\circ)
+\deg_{x_0}((x_{n-3}=0)^\circ).
\end{equation}

Moreover, \((x_n=\infty)^\circ\), after changing coordinates, belongs to the cyclic pattern of period \(2\)
\[
(1/x=0)^\circ
\to
(1/y=0)^\circ
\to
(1/x=0)^\circ.
\]
If \(f_n\) denotes the degree of this component with respect to \(x_0\), then \(f_0=1, f_1=0, f_2=1, f_3=0,\ldots\), and \(f_n=(1+(-1)^n)/2\). Thus \eqref{eq:p1p1-infty-decomposition} gives
\begin{equation}
\label{eq:p1p1-infty-degree}
d_n=t_{n-1}+t_{n-3}+f_n.
\end{equation}

Combining \eqref{eq:p1p1-zero-degree} and \eqref{eq:p1p1-infty-degree}, we obtain
\begin{equation}
\label{eq:p1p1-T-recurrence}
t_n+t_{n-4}=t_{n-1}+t_{n-3}+f_n.
\end{equation}
The initial values are
\[
t_0=1,
\qquad t_1=1,
\qquad t_2=2,
\qquad t_3=3.
\]
Solving this linear recurrence gives
\[
t_n=
\frac{1}{72}
\left(
6n^2+36n+47+9(-1)^n
+16\cos\frac{2\pi n}{3}
\right).
\]
Therefore
\[
d_n
=
\frac{1}{36}
\left(
6n^2+12n+23+9(-1)^n
-8\cos\frac{2\pi(n+1)}{3}
\right).
\]

Instead of using \((x_n=0)\) and \((x_n=\infty)\), one may also use \((x_n=1)\) to obtain the same degree sequence. In this case the singularity pattern gives
\begin{equation}
\label{eq:p1p1-basic-halburd-decomposition}
\deg_{x_0}(x_n=1)
=\deg_{x_0}((x_n=1)^\circ)
+2\deg_{x_0}((x_{n-2}=0)^\circ).
\end{equation}
Here the coefficient \(2\) comes from the multiplicity of the corresponding factor in the numerator of \(x_n-1\).

The component \((x_n=1)^\circ\) belongs to the cyclic pattern of period \(3\)
\[
(x-1=0)^\circ
\to
(y-1=0)^\circ
\to
(xy+1=0)^\circ
\to
(x-1=0)^\circ.
\]
If \(g_n\) denotes the degree of this component with respect to \(x_0\), then \(g_0=1, g_1=1, g_2=0, g_3=1,\ldots\), and
\[
g_n=
\frac{2}{3}
\left(
1-\cos\frac{2\pi(n+1)}{3}
\right).
\]
Thus we obtain
\begin{equation}
\label{eq:p1p1-one-degree}
d_n=2t_{n-2}+g_n.
\end{equation}
This is consistent with the expressions for \(d_n\) and \(t_n\) obtained above.

\section{Pullbacks and degree-drop divisors}
\label{sec:functoriality-divisor-degree}

Cancellations in iterated homogeneous representations arise when the image of a divisor enters the indeterminacy locus of the next rational map.

\begin{definition}[Pullback]
\label{def:pullback}
Let \(V,W\) be smooth projective varieties and let
\[
\psi:V\dashrightarrow W
\]
be a dominant rational map. Let \(D\) be a Cartier divisor on \(W\). Let
\[
\Gamma_\psi\subset V\times W
\]
be the closure of the graph over the open set on which \(\psi\) is defined, and choose a resolution of singularities
\[
\rho:\Gamma\to\Gamma_\psi.
\]
We denote the compositions with the two projections \(\Gamma_\psi\to V\) and \(\Gamma_\psi\to W\) by
\[
p:\Gamma\to V,
\qquad
q:\Gamma\to W.
\]
\[
\begin{array}{ccc}
& \Gamma & \\
p\swarrow && \searrow q\\
V && W .
\end{array}
\]
We define
\[
\psi^*D:=p_*q^*D
\]
and call it the pullback of \(D\) by \(\psi\).
\end{definition}

For an irreducible closed subset \(C\subset V\), if \(C\not\subset I(\psi)\), we set
\[
\psi_\circ(C):=\overline{\psi(C\setminus I(\psi))}.
\]
This is the closure of the image of the generic point of \(C\) under \(\psi\). Since the indeterminacy locus of a rational map has codimension at least \(2\), \(\psi_\circ(C)\) is always defined for a prime divisor \(C\subset V\).

Since \(V\) is smooth, this pullback is a Cartier divisor on \(V\).

For rational maps, in general the equality
\[
f^*g^*D=(g\circ f)^*D
\]
does not hold for compositions. This failure is observed as the cancellation of common factors in homogeneous representations. Such a failure of functoriality is closely related to algebraic stability criteria and to the behavior of pullback actions under iteration.

\begin{definition}[Functoriality defect]
\label{def:functoriality-defect}
Let $f:X\dashrightarrow Y$ and $g:Y\dashrightarrow Z$
be dominant rational maps of smooth projective varieties,
and let $D$ be a Cartier divisor on $Z$.
The functoriality defect of the composition is
the Cartier divisor on $X$
\begin{equation}
\label{eq:functoriality-with-common-factor}
K_{g\circ f}(D)
:=f^*g^*D-(g\circ f)^*D.
\end{equation}
\end{definition}

For example, consider $f=A\circ\sigma$ on $\PP^3$
from Section~\ref{example:working-cremona-degree-drop},
and let $g=\sigma$.
Since the greatest common divisor of the homogeneous components of $g\circ f$ is $x_0$,
for any hyperplane $D$ on $\PP^3$, we have
\[
K_{g\circ f}(D)=H_0=(x_0=0).
\]

\begin{remark}[Necessity of global generation]
\label{rem:necessity-global-generation}
The divisor $K_{g\circ f}(D)$ need not be effective.
Let
\[
\pi:\widetilde{\PP^2}\longrightarrow\PP^2
\]
be the blow-up of a point $p$, with exceptional
divisor $E$. Set
\[
X=Z=\widetilde{\PP^2},\qquad
Y=\PP^2,\qquad
f=\pi,\qquad g=\pi^{-1},
\]
and take $D=E$ on $Z$.
Since $g$ maps $\PP^2\setminus\{p\}$ into
$Z\setminus E$, we have $g^*E=0$.
On the other hand, $g\circ f=\mathrm{id}_X$.
Therefore,
\[
K_{g\circ f}(E)
=f^*g^*E-(g\circ f)^*E
=-E.
\]
Thus $K_{g\circ f}(D)$ can have negative
coefficients even for an effective Cartier divisor.
\end{remark}

In this section we will see that
$K_{g\circ f}(D)$ is effective whenever
$\OO_Z(D)$ is globally generated,
in particular when $D$ is very ample.

\subsection{The case of projective and multiprojective spaces}
\label{subsec:coordinate-degree-drop}

\paragraph{Projective space.}
Let $X=Y=Z=\PP^N$, with homogeneous coordinates
\[
\mathbf{x}=[x_0:\cdots:x_N],\quad
\mathbf{y}=[y_0:\cdots:y_N],\quad
\mathbf{z}=[z_0:\cdots:z_N],
\]
respectively.

Let $f:X\dashrightarrow Y$ and $g:Y\dashrightarrow Z$
be dominant rational maps represented by
\[
\begin{aligned}
f(\mathbf{x})
&=[F_0(\mathbf{x}):\cdots:F_N(\mathbf{x})],\\
g(\mathbf{y})
&=[G_0(\mathbf{y}):\cdots:G_N(\mathbf{y})].
\end{aligned}
\]
Here $F_i(\mathbf{x})$ are homogeneous polynomials
of the same degree $d_f$ in $x_0,\ldots,x_N$,
and $G_i(\mathbf{y})$ are homogeneous polynomials
of the same degree $d_g$ in $y_0,\ldots,y_N$.
Each tuple has no nonconstant common factor.

Substituting $y_j=F_j(\mathbf{x})$ into
$G_i(\mathbf{y})$, we obtain the homogeneous
polynomials on $X$
\[
U_i(\mathbf{x})
=G_i(F_0(\mathbf{x}),\ldots,F_N(\mathbf{x})),
\qquad 0\leq i\leq N.
\]
Let
\[
R(\mathbf{x})
=\gcd(U_0(\mathbf{x}),\ldots,U_N(\mathbf{x})).
\]
Then the reduced homogeneous representation of
the composition is
\[
(g\circ f)(\mathbf{x})
=[U_0(\mathbf{x})/R(\mathbf{x}):\cdots:U_N(\mathbf{x})/R(\mathbf{x})].
\]

Let $D=H$ be a hyperplane on $Z$ defined by
\[
H=(\ell(\mathbf{z})=0),\qquad
\ell(\mathbf{z})=\sum_{i=0}^N\lambda_i z_i,
\]
where $\lambda_i$ are constants, not all zero.

The pullback $g^*H$ is the divisor on $Y$
defined by
\[
g^*H=
\bigl(\ell(G_0(\mathbf{y}),\ldots,
             G_N(\mathbf{y}))=0\bigr).
\]
Pulling it back further by $f$, we obtain
the divisor on $X$
\[
f^*g^*H=
\bigl(\ell(U_0(\mathbf{x}),\ldots,
             U_N(\mathbf{x}))=0\bigr).
\]
On the other hand, the pullback by the
reduced composition is
\[
(g\circ f)^*H=
\bigl(\ell(U_0(\mathbf{x})/R(\mathbf{x}),\ldots,
             U_N(\mathbf{x})/R(\mathbf{x}))=0\bigr).
\]

Since $\ell$ is linear, we have
\[
\ell(U_0(\mathbf{x}),\ldots,U_N(\mathbf{x}))
=
R(\mathbf{x})
\ell(U_0(\mathbf{x})/R(\mathbf{x}),\ldots,U_N(\mathbf{x})/R(\mathbf{x})).
\]
Consequently,
\begin{equation}
\label{eq:projective-coordinate-defect}
K_{g\circ f}(H)
=f^*g^*H-(g\circ f)^*H
=(R(\mathbf{x})=0).
\end{equation}
Here $(R(\mathbf{x})=0)$ denotes the effective
zero divisor of the homogeneous polynomial
$R(\mathbf{x})$ on $X$.

Let $C\subset X=\PP^N$ be a prime divisor.
For a nonzero homogeneous polynomial $R(\mathbf{x})$,
let $\ord_C R$ denote the multiplicity of an irreducible
homogeneous defining polynomial of $C$ as a factor of $R$.

Since
\[
R=\gcd(U_0,\ldots,U_N),
\]
we have
\[
\ord_C R
=\min_{0\leq i\leq N}\ord_C U_i.
\]
Consequently,
\[
K_{g\circ f}(H)
=(R=0)
=\sum_{\ord_C R>0}(\ord_C R)C.
\]

The indeterminacy locus of $g$ is
\[
I(g)=V(G_0,\ldots,G_N)\subset Y.
\]
Since $F_0,\ldots,F_N$ have no nonconstant
common factor, $f$ is defined at the
generic point of every prime divisor
$C\subset X$.

Therefore,
\[
\begin{aligned}
\ord_C R>0
&\iff \ord_C U_i>0
       \quad\text{for all }i\\
&\iff G_i(F_0,\ldots,F_N)|_C=0
       \quad\text{for all }i\\
&\iff f_\circ(C)\subset I(g),
\end{aligned}
\]
where $f_\circ(C)$ is the closure of the
image of the generic point of $C$ under $f$.

Thus,
\begin{equation}
\label{eq:projective-defect-decomposition}
K_{g\circ f}(H)
=
\sum_{\substack{C\subset X\text{ prime}\\
f_\circ(C)\subset I(g)}}
\left(\min_{0\leq i\leq N}\ord_C U_i\right)C.
\end{equation}
This divisor is effective and independent of the chosen
hyperplane $H$. 

\begin{remark}
Taking degrees, we recover the classical
criterion of Forn\ae ss--Sibony \cite{FornaessSibony1994NATO};
see also \cite[Proposition~1.4.3]{Sibony1999}:
\[
\deg(g\circ f)=\deg g\cdot\deg f
\quad\Longleftrightarrow\quad
\nexists\, C\subset\PP^N\text{ prime divisor such that }
f_\circ(C)\subset I(g).
\]
\end{remark}

\paragraph{Iteration.}
Let $f:\PP^N\dashrightarrow\PP^N$ be a dominant rational self-map
with reduced homogeneous components $F_0,\ldots,F_N$.
Write the reduced representation of $f^n$ as
\[
f^n(\mathbf{x})
=[F_0^{(n)}(\mathbf{x}):\cdots:F_N^{(n)}(\mathbf{x})],
\qquad F_i^{(0)}(\mathbf{x})=x_i.
\]
At the next step, form the unreduced homogeneous polynomials
\[
U_i^{(n+1)}(\mathbf{x})
=F_i(F_0^{(n)}(\mathbf{x}),\ldots,F_N^{(n)}(\mathbf{x})),
\qquad
R_{n+1}=\gcd(U_0^{(n+1)},\ldots,U_N^{(n+1)}).
\]
The reduced components of $f^{n+1}$ are $U_i^{(n+1)}/R_{n+1}$.
For every hyperplane $H$ on the target copy of $\PP^N$,
\begin{equation}
\label{eq:projective-iterate-defect}
\begin{aligned}
K_{n+1}(H)&:=K_{f\circ f^n}(H)
=(R_{n+1}=0)\\
&=\sum_{C\subset X_0\text{ prime}}
\left(\min_{0\leq i\leq N}\ord_C U_i^{(n+1)}\right)C.
\end{aligned}
\end{equation}
Here and below, the prime divisors $C$ in the decomposition of
$K_{n+1}$ are on the initial space $X_0$.

\begin{example}[The working example on $\PP^3$]
\label{ex:p3-coordinate-iterate-K}
Consider $f=A\circ\sigma$ from
Section~\ref{example:working-cremona-degree-drop}, with the cubic
forms $F_0,\ldots,F_3$ given there. Put
\[
Q(\mathbf{x})=x_1x_2+x_1x_3+x_2x_3,
\qquad F_0(\mathbf{x})=x_0Q(\mathbf{x}).
\]
For the second iterate, the four components of $\sigma\circ f$
before cancellation are
\[
F_1F_2F_3,\quad F_0F_2F_3,\quad
F_0F_1F_3,\quad F_0F_1F_2.
\]
Their greatest common divisor is $x_0$. Since $A$ is invertible,
postcomposition by $A$ does not change the greatest common divisor.
Thus $R_2=x_0$. At the next step, substituting the reduced homogeneous representation
of $f^2$ into $F_0,\ldots,F_3$ gives $R_3=\gcd(U_0^{(3)},\ldots,U_3^{(3)})=Q(\mathbf{x})$.
Consequently, for any hyperplane $H$,
\begin{equation}
\label{eq:p3-coordinate-K2-K3}
K_2(H)=(x_0=0),\qquad
K_3(H)=(Q(\mathbf{x})=0).
\end{equation}
Note that
\[
f^*H_{0,1}=(x_0=0)+(Q(\mathbf{x})=0),
\]
but
\[
\begin{aligned}
(f^2)_\circ(x_0=0)
&=\{x_1=x_0,\ x_2=x_0+x_3\}
\not\subset I(f),\\
(f^2)_\circ(Q(\mathbf{x})=0)
&=\{[0:0:1:1]\}\subset I(f).
\end{aligned}
\]
Thus only the second component contributes to $K_3(H)$, and
so $d_2=3d_1-1=8$ and $d_3=3d_2-2=22$.
\end{example}

\paragraph{Products of projective lines.}
Let $X=Y=Z=(\PP^1)^M$, with homogeneous coordinate
tuples $\mathbf{x}=(\mathbf{x}_1,\ldots,\mathbf{x}_M)$,
$\mathbf{y}=(\mathbf{y}_1,\ldots,\mathbf{y}_M)$, and
$\mathbf{z}=(\mathbf{z}_1,\ldots,\mathbf{z}_M)$,
respectively, where
\[
\mathbf{x}_i=[x_{i,0}:x_{i,1}],\quad
\mathbf{y}_i=[y_{i,0}:y_{i,1}],\quad
\mathbf{z}_i=[z_{i,0}:z_{i,1}].
\]

Let $f:X\dashrightarrow Y$ and $g:Y\dashrightarrow Z$
be dominant rational maps represented in their
$i$-th coordinates by
\[
\begin{aligned}
f_i(\mathbf{x})
&=[F_{i,0}(\mathbf{x}):F_{i,1}(\mathbf{x})],\\
g_i(\mathbf{y})
&=[G_{i,0}(\mathbf{y}):G_{i,1}(\mathbf{y})].
\end{aligned}
\]
Each pair consists of coprime multihomogeneous
polynomials of the same multidegree.

Define the multihomogeneous polynomials on $X$
\[
U_{i,k}(\mathbf{x})
=G_{i,k}((F_{1,0}(\mathbf{x}),F_{1,1}(\mathbf{x})),\ldots,
          (F_{M,0}(\mathbf{x}),F_{M,1}(\mathbf{x}))),
\qquad k=0,1.
\]
Let
\[
R_i=\gcd(U_{i,0},U_{i,1}).
\]

For the coordinate divisor $H_i=(z_{i,1}=0)$
on $Z$, we have
\[
f^*g^*H_i=(U_{i,1}=0),\qquad
(g\circ f)^*H_i=(U_{i,1}/R_i=0).
\]
Consequently,
\begin{equation}
\label{eq:coordinate-pair-defect}
K_{g\circ f}(H_i)
=(R_i=0)
=\sum_{C\subset X\text{ prime}}
\left(\min_{k=0,1}\ord_C U_{i,k}\right)C.
\end{equation}
Here $\ord_C U_{i,k}$ is the multiplicity of
the defining irreducible multihomogeneous
polynomial of $C$ in $U_{i,k}$.

The indeterminacy loci are
\begin{equation}
\label{eq:coordinate-pair-indeterminacy}
I(g_i)=V(G_{i,0},G_{i,1}),\qquad
I(g)=\bigcup_{i=1}^M I(g_i).
\end{equation}
Since $f$ is defined at the generic point
of every prime divisor $C\subset X$,
\[
\min_{k=0,1}\ord_C U_{i,k}>0
\iff f_\circ(C)\subset I(g_i).
\]
Therefore,
\[
K_{g\circ f}(H_i)
=\sum_{f_\circ(C)\subset I(g_i)}\left(\min_{k=0,1}\ord_C U_{i,k}\right)C.
\]

\paragraph{Iteration.}
Let $f:(\PP^1)^M\dashrightarrow(\PP^1)^M$ be a dominant
rational self-map. Write its reduced coordinate pairs as
$[F_{i,0}:F_{i,1}]$, and those of $f^n$ as
$[F_{i,0}^{(n)}:F_{i,1}^{(n)}]$.
For each coordinate $i$, define
\[
U_{i,k}^{(n+1)}
=F_{i,k}((F_{1,0}^{(n)},F_{1,1}^{(n)}),\ldots,
                  (F_{M,0}^{(n)},F_{M,1}^{(n)})),
\qquad k=0,1.
\]
For the coordinate divisor $H_i$ on the target,
\begin{equation}
\label{eq:multiprojective-iterate-defect}
K_{n+1}(H_i)
=(\gcd(U_{i,0}^{(n+1)},U_{i,1}^{(n+1)})=0)
=\sum_{C\subset X_0\text{ prime}}
\left(\min_{k=0,1}\ord_C U_{i,k}^{(n+1)}\right)C.
\end{equation}
These divisors may differ with $i$.

\begin{example}[A coordinate-dependent degree drop]
\label{ex:dependence-ID}
Let $\varphi:(\PP^1)^2\dashrightarrow(\PP^1)^2$ be
\[
\varphi(x,y)=\left(x,\frac{x^2}{y}\right).
\]
Its first coordinate is a morphism; the second has the
indeterminacy points $(0,0)$ and $(\infty,\infty)$.
For the second iterate, substitution without cancellation gives
\[
y_2=\frac{x^2}{x^2/y}=\frac{x^2y}{x^2}.
\]
On the chart where $x,y$ are finite, the unreduced pair
(denominator, numerator) is $(x^2,x^2y)$, with common factor $x^2$.
Near $x=\infty$, setting $u=1/x$ gives the unreduced pair
$(u^2,u^2y)$, with common factor $u^2$.
The charts with $v=1/y$ give the same multiplicities and no
additional prime factors. Let $H_{\bar x}=(\bar x=0)$ and
$H_{\bar y}=(\bar y=0)$ be the coordinate divisors on the target.
\[
K_2(H_{\bar x})=0,\qquad K_2(H_{\bar y})=2(x=0)+2(x=\infty).
\]
Since $\varphi^2=\mathrm{id}$ as a rational map,
the calculation applies to all iterates:
\begin{equation}
\label{eq:involution-iterate-K}
K_{n+1}(H_{\bar x})=0,\qquad
K_{n+1}(H_{\bar y})=
\begin{cases}
0,& n\text{ even},\\
2(x=0)+2(x=\infty),&n\text{ odd}.
\end{cases}
\end{equation}
\end{example}

\begin{example}[The working example on $\PP^1\times\PP^1$]
\label{ex:p1p1-coordinate-iterate-K}
Consider the map from Section~\ref{ex:dp2},
\[
\varphi(x,y)=\left(1-\frac1x-y,\ x\right),
\qquad (x_n,y_n)=\varphi^n(x_0,y_0).
\]
Let $H_{\bar x}=(\bar x=0)$ and $H_{\bar y}=(\bar y=0)$ be the coordinate
divisors on the target. The second coordinate is $y_{n+1}=x_n$,
so the reduced coordinate pair for $y_{n+1}$ is already the pair
for $x_n$. Hence
$K_{n+1}(H_{\bar y})=0$ for all $n\geq0$.

For the first coordinate, write $x_m=A_m/B_m$ in reduced form,
starting from $(A_{-1},B_{-1})=(y_0,1)$ and
$(A_0,B_0)=(x_0,1)$. Since $y_m=x_{m-1}$, substitution gives
the unreduced numerator and denominator
\[
A_mB_{m-1}-B_mB_{m-1}-A_mA_{m-1}
\quad\text{and}\quad A_mB_{m-1}.
\]
After each step, cancel their common factors and use the
corresponding primitive pairs in the charts at infinity.
The first four steps have no common factor. At the fifth step,
the greatest common factor is $x_0^2$ in the finite initial chart,
with no additional factor in the other charts. Therefore
\begin{equation}
\label{eq:p1p1-coordinate-K5}
K_{n+1}(H_{\bar x})=0 \quad  \mbox{for }n=0,1,2,3, \qquad 
K_5(H_{\bar x})=2(x_0=0).
\end{equation}
The local calculation of the multiplicity two is also used in
Example~\ref{example:p1p1-degree-drop-orbit}.

\end{example}

\subsection{Functoriality criteria and divisor decompositions}

\paragraph{Very ample divisors.}

\begin{theorem}[Functoriality criterion for pullbacks]
\label{thm:functoriality-very-ample}
Let \(X,Y,Z\) be smooth projective varieties, and let
\(f:X\dashrightarrow Y\), \(g:Y\dashrightarrow Z\) be dominant rational maps.
Let \(A\) be a very ample Cartier divisor on \(Z\).

Let
\[
\Phi_A:Z\hookrightarrow \PP^M
\]
be the closed embedding associated with \(A\), and write
\[
\Phi_A\circ g=[Q_0:\cdots:Q_M]
\]
as a homogeneous representation with no common prime factor. For each prime divisor \(C\) on \(X\), set
\[
\mu_C(A)
:=
\min_{0\le i\le M}\ord_C(Q_i\circ f).
\]
Then the coefficient of \(K_{g\circ f}(A)\) along \(C\) is \(\mu_C(A)\). Moreover,
\[
\mu_C(A)>0
\quad\Longleftrightarrow\quad
f_\circ(C)\subset I(g).
\]

Consequently, if functoriality fails, namely if
\[
f^*g^*A\ne (g\circ f)^*A,
\]
then \(K_{g\circ f}(A)\) is a nonzero effective divisor. In this case we call \(K_{g\circ f}(A)\) the degree-drop divisor associated with \(A\) in the composition \(g\circ f\). It can be written as
\[
K_{g\circ f}(A)=\sum_{f_\circ(C)\subset I(g)} \mu_C(A)C.
\]
The sum ranges over all prime divisors \(C\) on \(X\) satisfying \(f_\circ(C)\subset I(g)\), and each \(\mu_C(A)\) is a positive integer.

In particular, the following two conditions are equivalent.
\begin{enumerate}[label=\textup{(\roman*)}]
\item
\[
f^*g^*A=(g\circ f)^*A.
\]
\item There is no prime divisor \(C\subset X\) satisfying \(f_\circ(C)\subset I(g)\).
\end{enumerate}
\end{theorem}

\begin{proof}
Since \(A\) is very ample, choose sections
\[
s_0,\ldots,s_M\in H^0(Z,\OO_Z(A))
\]
which give the closed embedding
\[
\Phi_A:Z\hookrightarrow \PP^M,\qquad
\Phi_A=[s_0:\cdots:s_M].
\]
Then
\[
\Phi_A^*\OO_{\PP^M}(1)\simeq \OO_Z(A),
\]
and, since \(\Phi_A\) is a closed embedding,
\[
I(\Phi_A\circ g)=I(g).
\]

By the pullback construction, we may write
\[
\Phi_A\circ g=[Q_0:\cdots:Q_M],
\]
where
\[
Q_0,\ldots,Q_M\in H^0(Y,\OO_Y(g^*A))
\]
have no common prime divisor on \(Y\).

Let \(C\subset X\) be a prime divisor.  Since the generic point of \(C\) is not
contained in \(I(f)\), the map \(f\) is regular at the generic point of \(C\).
Choose a local frame \(e\) of \(\OO_Y(g^*A)\) near its image.
Then
\[
\ord_C(Q_i\circ f):=\ord_C((Q_i/e)\circ f)
\]
is well-defined.

The pullback by the composite \(\Phi_A\circ g\circ f\) is obtained from
\[
[Q_0\circ f:\cdots:Q_M\circ f]
\]
by removing its common factors in codimension \(1\).  Therefore the coefficient
of \(K_{g\circ f}(A)\) along \(C\) is
\[
\mu_C(A)=\min_{0\leq i\leq M}\ord_C(Q_i\circ f).
\]

It remains to characterize the prime divisors \(C\) for which \(\mu_C(A)>0\). Since the generic point of a prime divisor \(C\subset X\) is not contained in \(I(f)\), \(f_\circ(C)\) is defined. By the formula above, \(\mu_C(A)>0\) is equivalent to
\[
\ord_C(Q_i\circ f)>0
\]
for all \(i\). This means that, along the image of the generic point of \(C\),
\[
Q_0=\cdots=Q_M=0,
\]
or equivalently
\[
f_\circ(C)\subset I(\Phi_A\circ g).
\]
Since \(I(\Phi_A\circ g)=I(g)\), we obtain
\[
\mu_C(A)>0
\quad\Longleftrightarrow\quad
f_\circ(C)\subset I(g).
\]

Thus, if functoriality fails, then
\[
K_{g\circ f}(A)=\sum_{f_\circ(C)\subset I(g)}\mu_C(A)C,
\]
and all coefficients are positive. Hence \(K_{g\circ f}(A)\) is a nonzero effective divisor. Conversely, if there is no prime divisor satisfying \(f_\circ(C)\subset I(g)\), then \(\mu_C(A)=0\) for all \(C\), and hence \(K_{g\circ f}(A)=0\). This proves the equivalence of \textup{(i)} and \textup{(ii)}.
\end{proof}

The following functoriality criterion for arbitrary Cartier divisors follows immediately from the theorem.

\begin{corollary}[Functoriality for arbitrary Cartier divisors]
\label{cor:all-cartier-smooth}
Under the notation of Theorem~\ref{thm:functoriality-very-ample}, the following conditions are equivalent.
\begin{enumerate}[label=\textup{(\roman*)}]
\item There is no prime divisor \(C\subset X\) satisfying \(f_\circ(C)\subset I(g)\).
\item For every Cartier divisor \(D\) on \(Z\), one has \(f^*g^*D=(g\circ f)^*D\).
\item For some very ample Cartier divisor \(A\) on \(Z\), one has \(f^*g^*A=(g\circ f)^*A\).
\end{enumerate}
\end{corollary}

\begin{proof}
The equivalence \((i)\Leftrightarrow(iii)\) is Theorem~\ref{thm:functoriality-very-ample}. The implication \((ii)\Rightarrow(iii)\) is clear. It remains to prove \((i)\Rightarrow(ii)\).

Take any Cartier divisor \(D\). Fix a very ample Cartier divisor \(H\) on \(Z\). For \(m\gg 0\), both \(mH\) and \(mH+D\) are very ample. By assumption \((i)\) and Theorem~\ref{thm:functoriality-very-ample}, we have \(f^*g^*(mH)=(g\circ f)^*(mH)\) and \(f^*g^*(mH+D)=(g\circ f)^*(mH+D)\). Since pullback is linear in Cartier divisors, subtracting the two equalities gives the assertion.
\end{proof}

\paragraph{Globally generated Cartier divisors.}

\begin{definition}[Divisor-dependent indeterminacy locus]
\label{def:divisor-dependent-indeterminacy}
Let \(g:Y\dashrightarrow Z\) be a dominant rational map of smooth projective
varieties, and let \(D\) be an effective Cartier divisor on \(Z\) such that
\(L:=\OO_Z(D)\) is \textbf{globally generated (equivalently, \(|D|\) is
base-point-free)}. We define
\begin{equation}
\label{eq:cartier-ID-linear-series}
I_D(g):=
\bigcap_{D'\in |D|}\Supp(g^*D')\subset Y.
\end{equation}
Here \(g^*D'\) denotes the pullback as a Cartier divisor, as in
Definition~\ref{def:pullback}, and \(\Supp(g^*D')\) is its underlying
closed set. Thus \(I_D(g)\) is the base locus of the linear system on
\(Y\) obtained by pulling back the members of \(|D|\).
\end{definition}

Equivalently, if $\Phi_{|D|}:Z\to\PP^r$ is the morphism defined by
the globally generated linear system $|D|$, then
\begin{equation}
\label{eq:ID-as-composite-indeterminacy}
I_D(g)=I(\Phi_{|D|}\circ g).
\end{equation}
Indeed, the primitive local coefficients of a basis of $|D|$
give a homogeneous representation of this composite, and their
common zero locus is the base locus of the pulled-back linear system.
For $Z=\PP^N$ and a hyperplane $D=H$, the morphism
$\Phi_{|H|}$ is a projective linear isomorphism, so
$I_H(g)=I(g)$.  For $Z=(\PP^1)^M$ and $D=H_i=(z_i=0)$,
$\Phi_{|H_i|}$ is the $i$-th projection, and therefore
$I_{H_i}(g)=I(g_i)=V(G_{i,0},G_{i,1})$ as in
\eqref{eq:coordinate-pair-indeterminacy}.

We regard the global sections of
$L=\OO_Z(D)$ as rational functions on $Z$.
Since $D$ is effective, we may choose a basis
\[
s_0=1,\quad s_1,\ldots,s_N
\]
of $H^0(Z,L)$.

Since $D+\operatorname{div}(s_i)$ is effective,
its pullback by $g$ is also effective. Thus
\[
g^*D+\operatorname{div}(s_i\circ g)
=g^*(D+\operatorname{div}(s_i))\geq 0.
\]
Hence each $s_i\circ g$ defines a global section
of $M=\OO_Y(g^*D)$.

Let $V\subset Y$ be an open set on which
$g^*D$ has a local equation $b=0$.
Then $1/b$ is a local frame of $M$, and
we obtain regular functions
$a_0,\ldots,a_N$ on $V$ by
\[
a_i=b(s_i\circ g)
\quad \mbox{ with }  \quad a_0=b.
\]
Here, the local equation $b$ is unique up to
multiplication by a unit. Such a change
multiplies all $a_i$ by the same unit
and therefore does not change the ideal
$(a_0,\ldots,a_N)$.

Since every nonzero global section of \(L\) is a linear combination of
the \(s_i\), definition~\eqref{eq:cartier-ID-linear-series} becomes
\begin{equation}
\label{eq:cartier-ID-local}
I_D(g)\cap V=V(a_0,\ldots,a_N).
\end{equation}
Since \(L\) is globally generated, a point where \(g\)
is regular cannot belong to \(I_D(g)\); hence
\[
I_D(g)\subset I(g).
\]
In particular, \(I_D(g)\) has codimension at least two whenever it is
nonempty.

\begin{theorem}[Decomposition of the functoriality defect for base-point-free Cartier divisors]
\label{thm:cartier-functoriality}
Let $X,Y,Z$ be smooth projective varieties, and let
$f:X\dashrightarrow Y$ and $g:Y\dashrightarrow Z$ be dominant
rational maps. Let $D$ be an effective Cartier divisor on $Z$ with
$\OO_Z(D)$ globally generated.
Choose a basis $s_0=1,s_1,\ldots,s_N$ of
$H^0(Z,\OO_Z(D))$.

For each prime divisor $C\subset X$, let $V$ be a neighborhood
of the generic point of $f_\circ(C)$, and let $b$ be a local
equation of $g^*D$ on $V$. As above, define the regular functions
\[
a_i=b(s_i\circ g),\qquad 0\leq i\leq N.
\]
Set
\begin{equation}
\label{eq:cartier-muC}
\mu_C(D):=\min_{0\leq i\leq N}\ord_C(a_i\circ f).
\end{equation}
This number is independent of the basis $s_i$,
the local equation $b$, and the neighborhood $V$.

Then the coefficient of the functoriality defect
\eqref{eq:functoriality-with-common-factor} along $C$ satisfies
\begin{equation}
\label{eq:cartier-defect-divisor}
\operatorname{coeff}_C K_{g\circ f}(D)=\mu_C(D).
\end{equation}
Moreover,
\[
\mu_C(D)>0
\quad\Longleftrightarrow\quad
f_\circ(C)\subset I_D(g).
\]
Consequently, $K_{g\circ f}(D)$ is effective and admits the prime
divisor decomposition
\begin{equation}
\label{eq:cartier-defect-decomposition}
K_{g\circ f}(D)
=\sum_{f_\circ(C)\subset I_D(g)}\mu_C(D)C,
\end{equation}
where every displayed coefficient is positive. In particular, the
following conditions are equivalent:
\begin{enumerate}[label=\textup{(\roman*)}]
\item $f^*g^*D=(g\circ f)^*D$.
\item There is no prime divisor $C\subset X$ with
$f_\circ(C)\subset I_D(g)$.
\end{enumerate}
\end{theorem}

\begin{proof}
Since $X$ is smooth, the prime divisor $C$ is
locally defined by a single equation $t=0$
near its generic point. Moreover, $f$ is
regular at the generic point of $C$, so
the functions $a_i\circ f$ are regular there.
Since $s_0=1$, we have $a_0=b$, which is a local
equation of $g^*D$. Hence
\[
\operatorname{coeff}_C(f^*g^*D)
=\ord_C(a_0\circ f).
\]

By the definition of $\mu_C(D)$, we can write
\[
a_i\circ f=t^{\mu_C(D)}c_i,
\]
where the $c_i$ are regular near the generic
point of $C$ and do not vanish simultaneously
there. Removing the common factor
$t^{\mu_C(D)}$ gives a regular representation
of $\Phi_{|D|}\circ g\circ f$ at the generic
point of $C$. Since $s_0=1$ corresponds
to the divisor $D$, we obtain
\[
\begin{aligned}
\operatorname{coeff}_C((g\circ f)^*D)
&=\ord_C(c_0)\\
&=\ord_C(a_0\circ f)-\mu_C(D).
\end{aligned}
\]
Combining the two coefficient formulas gives
\[
\operatorname{coeff}_C K_{g\circ f}(D)
=\mu_C(D)\geq0.
\]
By \eqref{eq:cartier-ID-local}, the generic
point of $f_\circ(C)$ belongs to $I_D(g)$
if and only if all $a_i$ vanish there.
Since $f$ is regular at the generic point
of $C$, this is equivalent to
$\ord_C(a_i\circ f)>0$ for every $i$,
or $\mu_C(D)>0$. This proves the support
criterion and the divisor decomposition.
\end{proof}

\paragraph{Iteration.}

The main case of interest is a dominant rational self-map
$f:X\dashrightarrow X$ of a smooth projective variety.
Applying Theorem~\ref{thm:cartier-functoriality} to the
composition $f\circ f^n$ gives the following result.

\begin{corollary}[Degree drops in iteration]
\label{cor:iterate-degree-drop}
Let $X$ be a smooth projective variety and
$f:X\dashrightarrow X$ a dominant rational self-map.
Let $D$ be an effective Cartier divisor such that
$\OO_X(D)$ is globally generated, and let $n\geq0$.

Choose a basis $s_0=1,s_1,\ldots,s_N$ of
$H^0(X,\OO_X(D))$. For each prime divisor $C\subset X$,
let $b$ be a local equation of $f^*D$ near the
generic point of $(f^n)_\circ(C)$, and define
\[
a_i=b(s_i\circ f),\qquad 0\leq i\leq N.
\]
Set
\begin{equation}
\label{eq:mu_nC}
\mu_{n,C}(D)
:=\min_{0\leq i\leq N}\ord_C(a_i\circ f^n).
\end{equation}

Then $\mu_{n,C}(D)$ is the coefficient along $C$ of
\begin{equation}
\label{eq:iterate-defect-divisor}
K_{n+1}(D):=K_{f\circ f^n}(D)
=(f^n)^*f^*D-(f^{n+1})^*D.
\end{equation}
Moreover,
\[
\mu_{n,C}(D)>0
\quad\Longleftrightarrow\quad
(f^n)_\circ(C)\subset I_D(f).
\]
Consequently, $K_{n+1}(D)$ is effective and
\begin{equation}
\label{eq:iterate-defect-decomposition}
K_{n+1}(D)
=\sum_{(f^n)_\circ(C)\subset I_D(f)}
\mu_{n,C}(D)C.
\end{equation}
In particular, the equality
$(f^n)^*f^*D=(f^{n+1})^*D$ fails if and only if
there exists a prime divisor $C\subset X$ such that
\begin{equation}
\label{eq:iterate-defect-condition}
(f^n)_\circ(C)\subset I_D(f).
\end{equation}
\end{corollary}

\begin{proof}
Apply Theorem~\ref{thm:cartier-functoriality}
with its first map replaced by $f^n$ and its
second map by $f$.
\end{proof}

\section{Orbit graph varieties and divisorial components}

\subsection{Finite-window orbit graphs}

From now on, let \(X\) be a smooth projective variety and let \(f:X\dashrightarrow X\) be a birational self-map. We write \(X_n=X\) for the copy of \(X\) at time \(n\).

For a finite interval \(I=[a,b]\subset\mathbb Z\), we define the finite-window orbit graph variety, or simply the orbit graph,
\[
\Omega_I \subset X_a\times X_{a+1}\times\cdots\times X_b,
\]
as the Zariski closure of orbit segments \((P_a,P_{a+1},\ldots,P_b)\) starting from \(P_a\in X_a\) for which \(P_{i+1}=f(P_i)\) \((i=a,\ldots,b-1)\) is defined at every step.

We denote by \(p_n:\Omega_I\to X_n\) the projection to the factor at time \(n\). Since the finite product is projective, \(\Omega_I\) is also projective. In general, however, it need not be smooth or normal.

As a basic property, on the dense open subset of \(\Omega_I\) consisting of genuine orbit segments, we have 
\[ p_n=f^k\circ p_{n-k} \]
for \(0\le k\le n-a\). Hence the same equality holds as an equality of rational maps on \(\Omega_I\). For a prime divisor, the corresponding statement concerns the component on the orbit graph determined by following its generic point; it should not be interpreted as a functoriality statement for pullbacks.

\subsection{Prime divisors and centers on the orbit graph}

Since \(p_n:\Omega_I\to X_n\) is a morphism, a Cartier divisor \(D_n\) on \(X_n\) gives a Cartier divisor \(p_n^*D_n\) on the orbit graph. On the other hand, points and lower-dimensional subvarieties appearing in the singularity patterns of Section 2 are usually not divisors on \(X_n\). For instance, the point \([0:0:1:1]\) and the line \(x_2=x_3=0\) in \(\PP^3\) are not divisors on \(\PP^3\). They do, however, appear as images at time \(n\) of prime Weil divisors on the orbit graph. We treat them through prime Weil divisors on the orbit graph.

Since \(\Omega_I\) is not necessarily normal, we take its normalization
\[
\nu:\Omega_I^\nu\to\Omega_I.
\]
For a prime Weil divisor, that is, an irreducible codimension-one closed subvariety
\[
E\subset \Omega_I^\nu,
\]
the valuation \(\ord_E\) defines a divisorial valuation on the function field
\[
K(\Omega_I^\nu)=K(\Omega_I).
\]
It assigns to a rational function on \(\Omega_I^\nu\) its order along \(E\).

We define the center of the prime Weil divisor \(E\) at time \(n\) by
\begin{equation}
\label{eq:center-time-n}
C_n(E):=\overline{(p_n\circ \nu)(E)}\subset X_n.
\end{equation}
This center need not be a divisor on \(X_n\); it may be a point or a lower-dimensional subvariety.

Conversely, from a prime divisor \(C\subset X_m\) on the copy at time \(m\),
one can obtain a prime Weil divisor on the orbit graph.  Indeed,
\(p_m\circ\nu:\Omega_I^\nu\to X_m\) is birational and hence is an isomorphism
over the generic point of \(C\).  We denote by
\begin{equation}
\label{eq:prime-divisor-from-time-m}
E(C,m)\subset \Omega_I^\nu
\end{equation}
the Zariski closure of the image of the generic point of \(C\) under
\((p_m\circ\nu)^{-1}\).  This is a prime Weil divisor on
\(\Omega_I^\nu\), and \(C_m(E(C,m))=C\).

Many of the divisorial components explicitly tracked in this paper are obtained
in this way.  We use the superscript \({}^{\circ}\) to denote the component
determined by the orbit of the generic point of a divisor specified at some
time.  Thus, if \(C\subset X_m\) is a prime divisor, then \(C^\circ\) denotes
this component; as a prime Weil divisor on the normalized orbit graph, it is
\(E(C,m)\).

The same component may also be described by its center at another time.  More
precisely, if
\[
C_\ell(E(C,m))=C'
\]
is a prime divisor on \(X_\ell\), then \(C^\circ\) and \((C')^\circ\) denote
the same component.  For example, in Section~\ref{ex:dp2}, the component whose
center at time \(n\) is \((x_n=0)\) has center \((y_{n+5}=0)\) at time
\(n+5\), and hence
\[
(y_{n+5}=0)^\circ=(x_n=0)^\circ .
\]

We next define the decomposition coefficient of a Cartier divisor \(D_n\) on \(X_n\) along a prime Weil divisor on the orbit graph. The pullback
\[
\nu^*p_n^*D_n
\]
is a Cartier divisor on \(\Omega_I^\nu\). Viewing it as a Weil divisor, we define the coefficient along a prime Weil divisor \(E\) by
\begin{equation}
\label{eq:def-alpha-n}
\alpha_n(E;D_n)
:=
\ord_E(\nu^*p_n^*D_n).
\end{equation}
Locally, if \(D_n\) is given on an open set \(U\subset X_n\) by a local equation \(g\), then
\[
\alpha_n(E;D_n)
=
\ord_E(g\circ p_n\circ\nu).
\]
Changing the local equation \(g\) only multiplies it by a unit, and its pullback to \(\Omega_I^\nu\) is also a unit near the generic point of \(E\). Hence this value is well-defined.

Since \(\Omega_I^\nu\) is a normal projective variety, for a Cartier divisor \(D_n\) on \(X_n\) we have a decomposition into prime Weil divisors
\begin{equation}
\label{eq:pullback-weil-decomposition}
\nu^*p_n^*D_n
=
\sum_E \alpha_n(E;D_n)E.
\end{equation}
Here \(E\) ranges over the prime Weil divisors on \(\Omega_I^\nu\).

We distinguish the components for which $C_j(E)$ is a divisor on $X_j$
for at least one $j\in I$ from those whose centers have codimension
at least two at every time in $I$.  Denote by $R_{I,n}(D_n)$ the sum,
with their coefficients, of all components of the latter kind.  Then
\begin{equation}
\label{eq:orbit-graph-visible-remainder}
\nu^*p_n^*D_n
=\sum_{\substack{E:\ \exists j\in I\\ \operatorname{codim}_{X_j}C_j(E)=1}}
  \alpha_n(E;D_n)E+R_{I,n}(D_n).
\end{equation}
This is still the complete divisor equality on the normalized orbit
graph; in particular, the remainder is not discarded. A component in
the first sum may nevertheless have a center of codimension at least
two at time $0$.

In particular, suppose \(D_n\subset X_n\) is a prime divisor, and put
\[
D_n^\circ:=E(D_n,n),
\]
the component corresponding to orbit segments through its generic point.
Since \(p_n\circ\nu\) is birational and \(X_n\) is normal, it is an
isomorphism over the generic point of \(D_n\).  Consequently,
\begin{equation}
\label{eq:principal-component-decomposition}
\begin{aligned}
\alpha_n(D_n^\circ;D_n)&=1,\\
\alpha_n(E;D_n)>0,\quad E\ne D_n^\circ
&\Longrightarrow \operatorname{codim}_{X_n} C_n(E)\ge 2.
\end{aligned}
\end{equation}
Thus the term \(D_n^\circ\) occurs with coefficient one in
\eqref{eq:pullback-weil-decomposition}, and every other component
appearing there has its center in \(D_n\) of codimension at least two.

We call a decomposition of the pullback \(\nu^*p_n^*D_n\) of a time-indexed divisor into prime Weil divisors on the orbit graph, with components indicated by their centers at each time, a Halburd-type divisor decomposition. In Halburd-type notation, this decomposition is abbreviated using centers at time \(n\) as
\begin{equation}
\label{eq:halburd-divisor-decomposition}
D_n
=
D_n^\circ
+
\sum_{E\ne D_n^\circ}
\alpha_n(E;D_n)\,C_n(E).
\end{equation}
Here the symbols \(C_n(E)\) on the right-hand side do not mean divisors on
\(X_n\); they are notation for the corresponding prime Weil divisors \(E\)
on the normalized orbit graph, represented by their centers at time \(n\).
Thus \eqref{eq:halburd-divisor-decomposition} should be understood as a
shorthand for the prime-divisor case of the complete equality
\eqref{eq:pullback-weil-decomposition}, using
\eqref{eq:principal-component-decomposition}, on \(\Omega_I^\nu\),
not as a divisor equality on \(X_n\).  If several
prime Weil divisors with the same center appear, as in infinitely near
situations, they must be distinguished.

In general, not all components appearing in these sums are directly involved in degree drops. Depending on the choice of the divisor \(D_n\), components arising from ordinary divisor decompositions, or components of singularity patterns not directly related to degree drops, may also appear.

\subsection{Local computation of decomposition coefficients}
\label{subsec:local-computation-decomposition-coefficients}

The decomposition coefficients can be computed using local parameters, just as in the usual computation of singularity patterns. Suppose, for example, that \(E=E(C,m)\), and that near the generic point of \(C\subset X_m\), the subvariety \(C\) is defined by a local equation
\[
s=0.
\]
Put
\[
s=\varepsilon,
\]
and take the remaining local coordinates to be generic. Using the iteration relation, express local coordinates at time \(n\) as functions of \(\varepsilon\), say
\[
z_1(\varepsilon),\ldots,z_r(\varepsilon).
\]
If a local equation of \(D_n\) is
\[
g(z_1,\ldots,z_r)=0,
\]
and
\[
g(z_1(\varepsilon),\ldots,z_r(\varepsilon))
=
\varepsilon^\alpha u(\varepsilon),
\qquad
u(0)\neq0,
\]
then the exponent \(\alpha\) is
\[
\alpha_n(E;D_n).
\]
This value can be computed locally at the generic point of \(C\), without explicitly constructing the normalization. The normalization is used to describe this value as the coefficient along a prime Weil divisor on the orbit graph.

For example, in the basic example on \(\PP^3\) in Section~\ref{example:working-cremona-degree-drop}, equation \eqref{eq:p3-basic-halburd-decomposition} displays the components contributing to degree, not necessarily the complete orbit-graph divisor decomposition. The terms on the right-hand side do not represent divisors on \(X_n\); rather, they denote the corresponding prime Weil divisors on the orbit graph by their centers at the indicated times. In particular, \((x_{0,n-1}=0)^\circ\) means that the prime Weil divisor determined by the generic point of the hyperplane at time \(n-1\) appears in the pullback of the divisor \((x_{0,n}=0)\) at time \(n\).

Similarly, in the basic example on \(\PP^1\times\PP^1\) in Section~\ref{ex:dp2}, for the divisor \((x_n=1)\) at time \(n\), we used the degree-relevant decomposition \eqref{eq:p1p1-basic-halburd-decomposition} in which the coefficient \(2\) appears. Indeed, put \(x_{n-2}=\varepsilon\) at time \(n-2\) and keep \(y_{n-2}\) generic. Then
\[
x_n=1+(1-y_{n-2})\varepsilon^2+O(\varepsilon^3).
\]
Since \(y_{n-2}\) is generic, the coefficient \(1-y_{n-2}\) is nonzero. Hence \(\ord_{(x_{n-2}=0)^\circ}(x_n-1)=2\), and \((x_{n-2}=0)^\circ\) appears with coefficient \(2\) in \eqref{eq:p1p1-basic-halburd-decomposition}.

\section{Linear relations for degrees}

In this section we organize the linear relations appearing in the divisor-orbit decomposition method using divisor degrees. The first kind of relation is obtained from degree-drop divisors. The second kind is obtained from Halburd-type divisor decompositions of time-indexed divisors.

From now on, the degrees of time-indexed divisors and prime Weil divisors on the orbit graph are always measured after pulling them back to the space \(X_0\simeq X\) at time \(0\). With this convention, we simply write \(\deg\) without an index.

\subsection{Two kinds of degree relations}
\label{subsec:divisor-degree}

\begin{definition}[Degree vector of a divisor]
\label{def:divisor-degree}
Let \(X\) be a smooth projective variety.
Let \(N^1(X)_{\mathbb Z}\) denote the lattice of numerical equivalence
classes of Cartier divisors on \(X\). Fix a \(\mathbb Z\)-basis
\[
\mathcal B=(B_1,\ldots,B_\rho)
\]
of \(N^1(X)_{\mathbb Z}\).

For a Cartier divisor \(D\) on \(X\), denote by \([D]_{\rm num}\) its
class in \(N^1(X)_{\mathbb Z}\). Write
\[
[D]_{\rm num}
=
d_1B_1+\cdots+d_\rho B_\rho
\]
with \(d_1,\ldots,d_\rho\in\mathbb Z\). We define
\[
\deg_{\mathcal B}(D):=(d_1,\ldots,d_\rho)\in\mathbb Z^\rho.
\]
We call \(\deg_{\mathcal B}(D)\) the degree vector of \(D\) with respect
to \(\mathcal B\).
\end{definition}

Let \(D_n \subset X_n=X \) be a Cartier divisor. We define
\begin{equation}
\label{eq:degree-of-time-n-divisor}
\deg(D_n)
:=
\deg_{\mathcal B}\bigl((f^n)^*D_n\bigr).
\end{equation}
Thus \(\deg(D_n)\) is not the degree measured directly on \(X_n\), but the degree measured after pulling back to time \(0\).

For a prime Weil divisor \(E\) on the normalization of the orbit graph, we define
\begin{equation}
\label{eq:degree-of-orbit-divisor}
\deg(E)
:=
\begin{cases}
\deg_{\mathcal B}(C_0(E)),&
C_0(E)\text{ is a divisor on }X_0,\\
0,&
\operatorname{codim}_{X_0}C_0(E)>1.
\end{cases}
\end{equation}
With the same convention, \(\deg(D_n^\circ)\) is the degree of the prime Weil divisor \(D_n^\circ\) on the orbit graph, measured through its center at time \(0\).
This degree definition can be extended additively to Weil divisors on the
normalized orbit graph.

\begin{remark}
The basis $\mathcal B$ used to define degree vectors
is arbitrary; its elements need not have base-point-free
representatives.
The functoriality defect $K_{g\circ f}(D)$ is defined
for any Cartier divisor $D$ and is additive in $D$.
Indeed, if $A$ is a very ample divisor, then for
$m\gg 0$ both $D+mA$ and $mA$ are very ample, and
\[
K_{g\circ f}(D)
=
K_{g\circ f}(D+mA)-K_{g\circ f}(mA).
\]
Thus the defect for an arbitrary Cartier divisor
can be computed from defects of globally generated
divisors. However, its coefficients need not be
nonnegative, as shown in
Remark~\ref{rem:necessity-global-generation}.
\end{remark}

\paragraph{Degree-drop relations}

We first describe the degree relation obtained from a degree-drop divisor. Let \(D_{n+1}\subset X_{n+1}\) be an effective Cartier divisor at time \(n+1\) such that \(\OO_{X_{n+1}}(D_{n+1})\) is globally generated. Applying Corollary~\ref{cor:iterate-degree-drop} with \(D=D_{n+1}\), equation \eqref{eq:iterate-defect-divisor} becomes
\begin{equation}
\label{eq:iterate-defect-divisor2}
\sum_{\substack{
C\\
(f^n)_\circ (C)\subset I_{D_{n+1}}(f)
}}
\mu_{n,C}(D_{n+1}) C=
 (f^n)^* f^*D_{n+1}
- (f^{n+1})^* D_{n+1}.
\end{equation}
Taking degrees gives
\begin{equation}
\label{eq:degree-drop-relation}
\sum_{\substack{
C\\
(f^n)_\circ (C)\subset I_{D_{n+1}}(f)
}}
\mu_{n,C}(D_{n+1})\deg(C)=
\deg\bigl(f^*D_{n+1}\bigr)
-
\deg(D_{n+1}).
\end{equation}
Here the sum ranges over all prime divisors \(C\) on \(X_0\) such that \((f^n)_\circ(C)\subset I_{D_{n+1}}(f)\), and \(\mu_{n,C}(D_{n+1})\) is given by \eqref{eq:mu_nC}.

\begin{example}[The degree-drop component in the $\PP^3$ working example]
\label{example:p3-degree-drop-orbit}

Return to the birational map $f=A\circ\sigma$ of
Section~\ref{example:working-cremona-degree-drop}.
For $n\geq1$, let $D_{n+1}=H_{n+1}$ be an arbitrary
hyperplane on $X_{n+1}=\PP^3$, and consider the
normalized orbit graph over $[0,n+1]$.

In the present case, since $H_{n+1}$ is a hyperplane
on $\PP^3$, we have
\[
I_{H_{n+1}}(f)=I(f)=I(\sigma).
\]
Moreover, the degree-drop divisor can be obtained directly
from the homogeneous coordinate formula
\eqref{eq:projective-iterate-defect} in Section~3.1,
without using the general Cartier-divisor theorem.
Indeed,
\[
K_{n+1}(H)
=(R_{n+1}=0)
=\sum_{C\subset X_0\text{ prime}}
\left(
\min_{0\leq i\leq3}\ord_C U_i^{(n+1)}
\right)C.
\]
Since $f=A\circ\sigma$ and $A$ is invertible,
the common factors of $U_0^{(n+1)},\ldots,U_3^{(n+1)}$
are those of the four triple products of the reduced
homogeneous components of $f^n$.
Consequently, a prime divisor $C$ contributes
if and only if
\[
(f^n)_\circ(C)\subset I(f)=I(\sigma).
\]

We now determine all such prime divisors.
The only divisors contracted by $f$ are the four
coordinate hyperplanes $(x_i=0)$, $i=0,1,2,3$.
As shown in
Section~\ref{example:working-cremona-degree-drop},
the images of $(x_i=0)$ for $i=1,2,3$ and their
subsequent orbits are never contained in $I(f)$.
On the other hand,
\[
f_\circ(x_0=0)=[0:0:1:1]\in I(f),
\]
whereas
\[
(f^2)_\circ(x_0=0)
=\{x_1=x_0,\ x_2=x_0+x_3\}
\not\subset I(f),
\]
and its subsequent images also avoid $I(f)$.

Thus a prime divisor can contribute to $K_{n+1}(H)$
only when its image at time $n-1$ is the coordinate
hyperplane $(x_{0,n-1}=0)$.
Let
\[
H_{0,n-1}^\circ=(x_{0,n-1}=0)^\circ
\]
be the corresponding prime divisor on the orbit graph.
For any prime divisor $C\subset X_0$, we therefore have
\[
\begin{aligned}
(f^n)_\circ(C)\subset I_H(f)
&\Longleftrightarrow
(f^{n-1})_\circ(C)
=(p_{n-1}\circ\nu)_*H_{0,n-1}^\circ\\
&\Longleftrightarrow
C=(p_0\circ\nu)_*H_{0,n-1}^\circ.
\end{aligned}
\]

The multiplicity of this component is one.
Indeed, along the generic point of $(x_{0,n-1}=0)$,
the first two homogeneous components of $f$ vanish
to first order, while the other two do not vanish.
The four triple products defining the next application
of $\sigma$ therefore have common vanishing order one.
Hence
\begin{equation}
\label{eq:p3-working-K-pushforward}
K_{n+1}(H)
=(p_0\circ\nu)_*H_{0,n-1}^\circ.
\end{equation}

For $n=1$ and $n=2$, this gives respectively
\[
K_2(H)=(x_{0,0}=0),\qquad
K_3(H)=
(x_{1,0}x_{2,0}+x_{1,0}x_{3,0}
+x_{2,0}x_{3,0}=0),
\]
in agreement with \eqref{eq:p3-coordinate-K2-K3}.

By the definition of the degree of an orbit-graph
prime divisor,
\[
\deg((p_0\circ\nu)_*H_{0,n-1}^\circ)
=\deg(H_{0,n-1}^\circ)=t_{n-1},
\]
including the convention that the degree is zero
when its center at time $0$ has codimension at least two.

Since $f$ is cubic, taking degrees in
\eqref{eq:p3-working-K-pushforward} gives
\[
3d_n-d_{n+1}=t_{n-1},
\]
or equivalently,
\[
d_{n+1}=3d_n-t_{n-1},
\]
which is \eqref{eq:cremona-degree-drop}.
\end{example}

\begin{example}[The degree-drop component in the $\PP^1\times\PP^1$ working example]
\label{example:p1p1-degree-drop-orbit}

Return to the birational map $f=\varphi$ of
Section~\ref{ex:dp2}.
For $n\geq4$, let $D_{n+1}=(x_{n+1}=a)$ be the copy
of $D=(x=a)$ on $X_{n+1}$, where $a$ is a generic
constant, and consider the normalized orbit graph
over $[0,n+1]$.

For generic $a$, the divisor $D_{n+1}$ avoids
the images of all exceptional prime divisors
of $p_{n+1}\circ\nu$. Consequently, on the
normalized orbit graph,
\[
\nu^*p_{n+1}^*D_{n+1}=D_{n+1}^\circ.
\]
Thus the Halburd-type divisor decomposition
of $D_{n+1}$ is trivial.

Since $\OO_{X_{n+1}}(D_{n+1})\simeq\OO(1,0)$,
the divisor-dependent indeterminacy locus is
\[
I_{D_{n+1}}(f)=I(f_1)=\{(0,\infty)\}.
\]
As in the $\PP^3$ example, the degree-drop divisor
can be computed directly from the coordinate formula
\eqref{eq:multiprojective-iterate-defect} in Section~3.1,
without using the general Cartier-divisor theorem.
Indeed, replacing the first target coordinate
$x_{n+1}$ by $x_{n+1}-a$ is an invertible linear
change of its homogeneous coordinate pair and
does not change its greatest common divisor.
Consequently,
\[
K_{n+1}(D)
=K_{n+1}(H_{\bar x})
=\sum_{C\subset X_0\text{ prime}}
\left(
\min_{k=0,1}\ord_C U_{1,k}^{(n+1)}
\right)C.
\]
Thus a prime divisor $C\subset X_0$ contributes
if and only if
\[
(f^n)_\circ(C)\subset I_{D_{n+1}}(f)
=\{(0,\infty)\}.
\]

We now determine all such prime divisors.
The only divisor contracted by $f$ is $(x=0)$.
As shown in Section~\ref{ex:dp2}, its singularity
pattern is
\[
(x=0)^\circ
\to(\infty,0)
\to(1,\infty)
\to(\infty,1)
\to(0,\infty)
\to(y=0)^\circ.
\]
Therefore a prime divisor can contribute to
$K_{n+1}(D)$ only when its image at time $n-4$
is the divisor $(x_{n-4}=0)$.

For any prime divisor $C\subset X_0$, we have
\[
\begin{aligned}
(f^n)_\circ(C)\subset I_{D_{n+1}}(f)
&\Longleftrightarrow
(f^{n-4})_\circ(C)
=(p_{n-4}\circ\nu)_*((x_{n-4}=0)^\circ)\\
&\Longleftrightarrow
C=(p_0\circ\nu)_*((x_{n-4}=0)^\circ).
\end{aligned}
\]

We next compute the multiplicity of this component.
At time $n-4$, put
\[
(x_{n-4},y_{n-4})=(\varepsilon,b),
\]
where $b$ is generic. Near $(0,\infty)$ at time $n$,
put
\[
u=x_n,\qquad v=\frac1{y_n}.
\]
Then
\[
x_{n+1}-a
=1-a-\frac1{x_n}-y_n
=\frac{(1-a)uv-u-v}{uv}.
\]
Following the singularity pattern gives
\[
u=-\varepsilon+O(\varepsilon^2),
\qquad
v=\varepsilon+O(\varepsilon^2).
\]
A further computation yields
\[
\begin{aligned}
(1-a)uv-u-v
&=(a-b)\varepsilon^2+O(\varepsilon^3),\\
uv&=-\varepsilon^2+O(\varepsilon^3).
\end{aligned}
\]
Since $a$ and $b$ are generic, $a-b\ne0$.
Thus the numerator and denominator have a common
factor of order two, and the multiplicity of
the degree-drop component is two.

Consequently,
\begin{equation}
\label{eq:p1p1-working-K-pushforward}
K_{n+1}(D)
=2(p_0\circ\nu)_*((x_{n-4}=0)^\circ).
\end{equation}
For $n=4$, this gives
\[
K_5(H_{\bar x})=2(x_0=0),
\]
in agreement with \eqref{eq:p1p1-coordinate-K5}.

By the definition of the degree of an orbit-graph
prime divisor, the degree of the component on
the right-hand side of
\eqref{eq:p1p1-working-K-pushforward},
measured with respect to $x_0$, is $t_{n-4}$.

Since $f^*D$ has bidegree $(1,1)$ and
$y_n=x_{n-1}$, taking degrees with respect
to $x_0$ gives
\[
d_n+d_{n-1}-d_{n+1}=2t_{n-4},
\]
or equivalently,
\[
d_{n+1}=d_n+d_{n-1}-2t_{n-4}.
\]

Thus the degree-drop divisor on $X_0$ is obtained
by pushing forward the unique contributing
orbit-graph component $(x_{n-4}=0)^\circ$,
with multiplicity two.
\end{example}

\paragraph{Halburd-type degree relations}

Let $D_n\subset X_n$ be a prime Cartier divisor, and let
$I$ be a time window containing $0$ and $n$.
Taking degrees in the Halburd-type divisor decomposition
\eqref{eq:halburd-divisor-decomposition}, we obtain
\begin{equation}
\label{eq:degree-from-halburd-decomposition}
\deg(D_n)
=
\deg(D_n^\circ)
+
\sum_{E\ne D_n^\circ}
\alpha_n(E;D_n)\deg(E).
\end{equation}
By definition, $\deg(E)=0$ whenever the center $C_0(E)$
has codimension at least two in $X_0$.
In particular, every component of $R_{I,n}(D_n)$
in \eqref{eq:orbit-graph-visible-remainder} has degree zero.
The same holds for components in the first sum of
\eqref{eq:orbit-graph-visible-remainder} whose centers
at time $0$ have codimension at least two.
These components do not contribute to the degree relation,
although they remain in the complete divisor decomposition.
This is the basic linear relation obtained from a Halburd-type
divisor decomposition. However, depending on the choice of $D_n$,
the right-hand side may involve additional divisor sequences
whose degrees must also be determined.

\begin{example}[An example in which extra divisor sequences appear]
Consider the map $f=A\circ\sigma$ from
Section~\ref{example:working-cremona-degree-drop},
and the hyperplane
\[
D_n=\{2x_{0,n}-x_{1,n}-x_{2,n}-x_{3,n}=0\}
\subset X_n.
\]
Since $D_n=A(H_0)$, we have
\[
f^*D_n
=(x_{1,n-1}=0)+(x_{2,n-1}=0)+(x_{3,n-1}=0).
\]
On the orbit graph, $D_n^\circ$ has center
$[1:0:0:0]$ at time $n-1$.
Hence, omitting zero-degree components, we obtain
\[
\deg(D_n)
=\deg(D_n^\circ)
+\sum_{i=1}^3\deg((x_{i,n-1}=0)^\circ).
\]

Thus, if one uses this \(D_n\), the degree relations also involve the degree sequences of \((x_{1,n-1}=0)^\circ\), \((x_{2,n-1}=0)^\circ\), and \((x_{3,n-1}=0)^\circ\). If these components can be described as periodic or finite-type singularity patterns, one can write linear relations including them. However, if the aim is to compute the degree drop arising from \(H_0=(x_0=0)\) with the fewest auxiliary variables, then this choice is inconvenient because it introduces extra divisor sequences. In such a situation it is more effective to choose a divisor whose decomposition closes with a small number of known divisor sequences, as in \eqref{eq:p3-basic-halburd-decomposition}.
\end{example}

\subsection{Finite-type condition and closed linear systems}

Whether the two kinds of degree relations described above yield a closed system of linear relations depends on whether the divisorial components appearing in them can be described, up to time shifts, by finitely many divisor sequences.

\begin{proposition}[Closed system of linear relations under a finite-type condition]
\label{prop:finite-type-closed-linear-system}
Let \(D\) be an effective Cartier prime divisor on \(X\) such that \(\OO_X(D)\) is globally generated, and let \(D_n\) be its copy at time \(n\). Assume that, among the divisorial components appearing in the Halburd-type divisor decomposition of \(D_n\) at each time \(n\) and in the degree-drop relations in the iteration, those contributing to degree can be expressed, up to time shifts, by finitely many divisor sequences
\[
E^{(1)}_n,\ldots,E^{(r)}_n.
\]
Assume further that the decomposition coefficients and the degree-drop coefficients are constant or periodic in \(n\), and that all relations can be written using only time shifts of these finitely many divisor sequences.

Then the degree sequence
\[
\deg(D_n)
\]
and the auxiliary degree sequences
\[
u^{(j)}_n:=\deg E^{(j)}_n
\qquad (j=1,\ldots,r)
\]
satisfy a closed system consisting of finitely many linear difference relations.
\end{proposition}

\begin{proof}
Equation \eqref{eq:degree-from-halburd-decomposition} gives a linear relation at each time \(n\). By assumption, the prime Weil divisors contributing to degree on the right-hand side are, up to time shifts, among the finitely many divisor sequences \(E^{(1)}_n,\ldots,E^{(r)}_n\). Hence the right-hand side can be written as a linear combination of finitely many auxiliary degree sequences
\[
u^{(j)}_{n+k}
=
\deg E^{(j)}_{n+k}.
\]

For degree drops in the iteration, equation \eqref{eq:degree-drop-relation} shows that the common-factor components appearing there are also, under the assumption, expressed by time shifts of the same finitely many divisor sequences. Thus this relation can also be written as a finite linear relation involving
\[
\deg(D_{n+1}),\ \deg(D_n),
\quad
u^{(j)}_{n+k}.
\]

Therefore the degree relations obtained from Halburd-type divisor decompositions and the degree-drop relations involve only
\[
\deg(D_n),
\quad
u^{(1)}_n,\ldots,u^{(r)}_n,
\]
and finitely many of their time shifts. Hence they form a closed system of linear difference relations among finitely many degree sequences.
\end{proof}

For example, in the basic example of Section~\ref{example:working-cremona-degree-drop}, if one chooses the divisor \(H_0=(x_0=0)\), then the only divisor sequence contributing to degree is the time shift of \((x_{0,n}=0)^\circ\), and the relations \eqref{eq:cremona-degree-decomposition} and \eqref{eq:cremona-degree-drop} give a closed linear system. On the other hand, if one chooses \(D=A(H_0)=(2x_0-x_1-x_2-x_3=0)\), then, as seen in the preceding example, the decomposition also contains components corresponding to the three coordinate hyperplanes. In this case, too, the system closes with finitely many linear relations if one also tracks the divisor sequences corresponding to each coordinate hyperplane, but the number of auxiliary degree sequences is larger than in the case of \(H_0\). Thus the choice of divisor affects the size of the resulting linear system.

In this way, even if the divisorial components appearing in degree drops can be described by finitely many divisor sequences, in order to obtain a closed system of linear relations one must choose divisors that detect these components appropriately. Hence the fact that the degree-drop components are described by finitely many divisor sequences does not immediately imply the existence of a divisor for which the Halburd-type divisor decompositions and the degree-drop relations together form a closed finite linear system. We do not pursue a general theory of this point in the present paper, and only pose the following problem.

\begin{problem}
In the divisor-orbit decomposition method, suppose that the divisorial components appearing in degree drops in the iteration can be described, up to time shifts, by finitely many divisor sequences. Is it then possible, by choosing suitable divisors \(D_n\) whose associated line bundles are globally generated, to obtain from the Halburd-type degree relations \eqref{eq:degree-from-halburd-decomposition} and the degree-drop relations \eqref{eq:degree-drop-relation} a linear system that closes with finitely many divisor sequences, up to time shifts, as unknowns?
\end{problem}

\section{An example with only short singularity patterns}

In this section we treat a second-order rational map considered by
Ramani--Grammaticos--Willox--Mase
\cite{RamaniGrammaticosWilloxMase2017}.  This map is integrable in the
sense that, after resolving the relevant base points, it preserves an
elliptic fibration; the associated rational elliptic surface has a
reducible singular fiber of type \(A_1^{(1)}\).  In that paper, this map
was used as an example for which applying Halburd-type counting directly
as a one-variable second-order dynamical system does not yield a closed
degree relation.  Here we organize this map as a valuation computation on
the orbit graph of \(\PP^1\times\PP^1\).  In this example, to obtain a
closed degree relation, it is necessary to consider not only divisors of
the type \(x=\) constant, but also suitable auxiliary divisors.

\subsection{Degree relations from Halburd-type divisor decompositions}

Let \(\varphi:(x,y)\mapsto (\bar x,\bar y)\) be the birational map on \(\PP^1\times\PP^1\) defined by
\[
  \frac{x+\bar x-2z}{x+\bar x}\cdot
  \frac{x+y-2z}{x+y}
  =\frac{N(x)}{D(x)},
  \qquad
  \bar y=x,
\]
where
\[
N(x)=((x-z)^2-a^2)((x-z)^2-b^2),
\qquad
D(x)=(x^2-c^2)(x^2-d^2).
\]
Here \(a,b,c,d,z\) are generic constants. We write the orbit as \((x_n,y_n)=\varphi^n(x_0,y_0)\).

\begin{figure}[htbp]
\centering
\includegraphics[width=.4\textwidth]{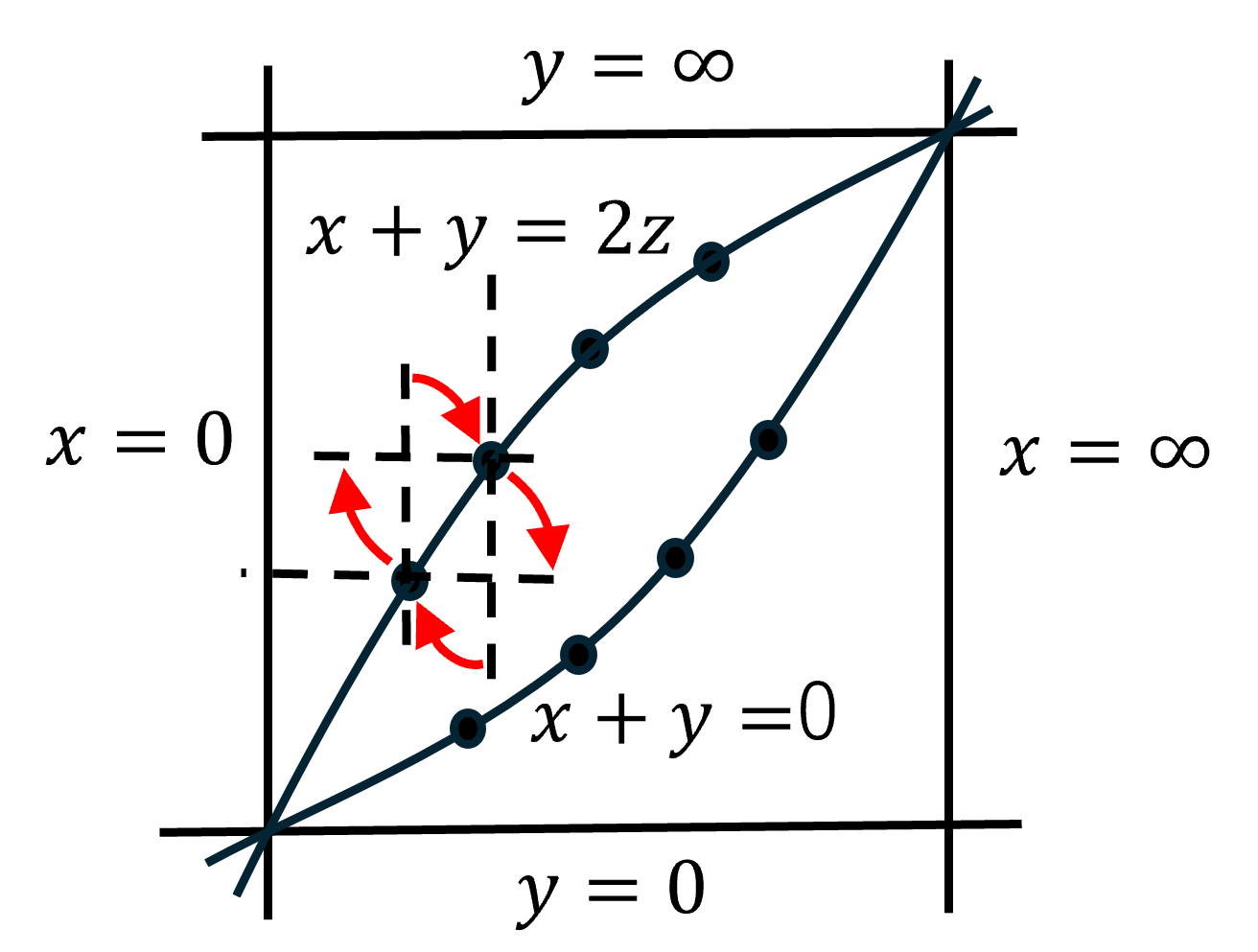}
\caption{The two \((1,1)\)-curves appearing in the example.  After
resolving the indeterminacy points, their union gives a reducible fiber
of the elliptic fibration on the associated rational elliptic surface.}
\label{fig:Ex6}
\end{figure}

The singularity patterns of this map consist only of the following four short symmetric types (Figure~\ref{fig:Ex6}):
\begin{align*}
(z\pm a,\ast)&\longrightarrow (z\mp a,z\pm a)
              \longrightarrow (\ast,z\mp a),\\
(z\pm b,\ast)&\longrightarrow (z\mp b,z\pm b)
              \longrightarrow (\ast,z\mp b),\\
(\pm c,\ast)&\longrightarrow (\mp c,\pm c)
              \longrightarrow (\ast,\mp c),\\
(\pm d,\ast)&\longrightarrow (\mp d,\pm d)
              \longrightarrow (\ast,\mp d).
\end{align*}

For example, the following terms of the pullback of \((x_n=z+a)\) on the normalized orbit graph are those contributing to bidegree:
\begin{equation}
\label{eq:decomposition-xn-za}
\operatorname{bideg}(x_n=z+a)
=\operatorname{bideg}((x_n=z+a)^\circ)
+\operatorname{bideg}((x_{n-1}=z-a)^\circ).
\end{equation}

Let
\[
\mathbf d_n:=\operatorname{bideg}x_n\in\mathbb Z^2
\]
be the bidegree of \(x_n\) with respect to \((x_0,y_0)\). Then the bidegree of \(y_n\) is \(\mathbf d_{n-1}\). Set
\[
\mathbf t_n:=\operatorname{bideg}(x_n=z+a)^\circ.
\]
By symmetry,
\[
\operatorname{bideg}(x_n=z\pm a)^\circ
=
\operatorname{bideg}(x_n=z\pm b)^\circ
=
\operatorname{bideg}(x_n=\pm c)^\circ
=
\operatorname{bideg}(x_n=\pm d)^\circ
=
\mathbf t_n.
\]
Therefore \eqref{eq:decomposition-xn-za} gives
\begin{equation}
\label{eq:dntntn}
\mathbf d_n=\mathbf t_n+\mathbf t_{n-1}.
\end{equation}

To obtain a closed formula for the degrees, consider the prime divisor
\(D:=\{x+y=2z\}\) on \(\PP^1\times\PP^1\). Since \(D\) contains the singular centers
\[
(z+a,z-a),\quad (z-a,z+a),\quad (z+b,z-b),\quad (z-b,z+b),
\]
the corresponding bidegree relation is
\begin{align}
\label{eq:auxiliary-divisor-decomposition-conic}
\operatorname{bideg}(x_n+y_n=2z)
&=\operatorname{bideg}((x_n+y_n=2z)^\circ) \nonumber\\
&\quad+\operatorname{bideg}((x_{n-1}=z-a)^\circ)
+\operatorname{bideg}((x_{n-1}=z+a)^\circ) \nonumber\\
&\quad+\operatorname{bideg}((x_{n-1}=z-b)^\circ)
+\operatorname{bideg}((x_{n-1}=z+b)^\circ).
\end{align}
Here the first term on the right belongs, on the orbit graph, to the divisor sequence of period \(2\)
\[
(x_n+y_n=2z)^\circ
\longrightarrow
(x_{n+1}+y_{n+1}=0)^\circ
\longrightarrow
(x_{n+2}+y_{n+2}=2z)^\circ.
\]
Thus
\[
\operatorname{bideg}(x_n+y_n=2z)^\circ=(1,1).
\]
Taking bidegrees in \eqref{eq:auxiliary-divisor-decomposition-conic}, we obtain
\[
\mathbf d_n+\mathbf d_{n-1}
=
(1,1)+4\mathbf t_{n-1}.
\]
Together with \eqref{eq:dntntn} and the initial values
\[
\mathbf t_0=(1,0),
\qquad
\mathbf t_1=(3,1),
\]
this gives
\[
\mathbf t_n=
\left(
\frac{(n+1)(n+2)}{2},
\frac{n(n+1)}{2}
\right)
\]
and
\[
\mathbf d_n=\bigl((n+1)^2,n^2\bigr).
\]

\subsection{The degree-drop relation}

We also record the degree-drop relation in this example using the notation of Section 5. Consider the divisor at time \(n+1\)
\[
D_{n+1}:=(x_{n+1}=0).
\]
Here \(\OO_{X_{n+1}}(D_{n+1})\simeq\OO(1,0)\) is globally generated.

Pulling \(D_{n+1}\) back one step to time \(n\), we obtain from
\[
\frac{x_n-2z}{x_n}
\frac{x_n+y_n-2z}{x_n+y_n}
=
\frac{N(x_n)}{D(x_n)}
\]
the equation
\[
N(x_n)x_n(x_n+y_n)
-
D(x_n)(x_n-2z)(x_n+y_n-2z)
=
0.
\]
Since the highest-degree terms cancel, this is a divisor of bidegree \((4,1)\) in the variables \((x_n,y_n)\) at time \(n\). Since \(y_n=x_{n-1}\), we have
\[
\operatorname{bideg}\bigl(\varphi^*D_{n+1}\bigr)
=
4\mathbf d_n+\mathbf d_{n-1}.
\]

On the other hand, the bidegree of the divisor obtained by pulling \(D_{n+1}\) directly back to time \(0\) is
\[
\operatorname{bideg}(D_{n+1})=\mathbf d_{n+1}.
\]

The prime divisors involved in the degree drop are the divisorial components whose centers at time \(n\) are contained in
\[
I_{D_{n+1}}(\varphi).
\]
In this example, they correspond to the intermediate points of the eight short singularity patterns:
\[
(z+a,z-a)_n,
\quad (z-a,z+a)_n,
\quad (z+b,z-b)_n,
\quad (z-b,z+b)_n,
\]
\[
(c,-c)_n,
\quad (-c,c)_n,
\quad (d,-d)_n,
\quad (-d,d)_n.
\]
Each of these is the principal component of a special fiber one time step earlier, for example
\[
(z+a,z-a)_n=(x_{n-1}=z-a)^\circ,
\qquad
(z-a,z+a)_n=(x_{n-1}=z+a)^\circ,
\]
and similarly for the others. Hence, by symmetry and the definition of \(\mathbf t_n\), the sum of their bidegrees is
\[
8\mathbf t_{n-1}.
\]

Thus, corresponding to the degree-drop relation \eqref{eq:degree-drop-relation} of Section 5, in this example we obtain
\[
4\mathbf d_n+\mathbf d_{n-1}
=
\mathbf d_{n+1}
+
8\mathbf t_{n-1}.
\]

\section{A four-dimensional multiplicative map}

In this section we treat the four-dimensional multiplicative map studied by Stokes--Takenawa--Carstea\cite{StokesTakenawaCarstea2025}. In this paper, we compute the degree growth of this map as multidegree relations obtained from Halburd-type divisor decompositions. As above, components of zero multidegree are not tracked.

Let \(\varphi:(x_1,x_2,x_3,x_4)\mapsto(\bar x_1,\bar x_2,\bar x_3,\bar x_4)\) be the birational map on \((\PP^1)^4\) defined by
\[
\bar x_1=\frac{1+bx_3}{x_1x_3x_4},
\qquad
\bar x_2=x_3,
\qquad
\bar x_3=\frac{1+bx_1}{x_1x_2x_3},
\qquad
\bar x_4=x_1.
\]
We write the orbit as
\[
(x_{1,n},x_{2,n},x_{3,n},x_{4,n})
=
\varphi^n(x_1,x_2,x_3,x_4).
\]

Let \(\operatorname{mdeg}\) denote the multidegree with respect to the
initial affine variables
\[
(x_{1,0},x_{2,0},x_{3,0},x_{4,0}),
\]
or equivalently the degree vector with respect to the standard basis of
\(N^1((\PP^1)^4)_{\mathbb Z}\).  We set
\[
\mathbf d_n:=\operatorname{mdeg} x_{1,n},
\qquad
\mathbf e_n:=\operatorname{mdeg} x_{3,n}.
\]
Since \(x_{2,n}=x_{3,n-1}\) and \(x_{4,n}=x_{1,n-1}\), we have
\[
\operatorname{mdeg} x_{2,n}=\mathbf e_{n-1},
\qquad
\operatorname{mdeg} x_{4,n}=\mathbf d_{n-1}.
\]

This map has confined singularity patterns starting from \(x_1=-b^{-1}\) and \(x_3=-b^{-1}\):
\begin{equation}
\label{eq:4d-conf-x1}
\begin{aligned}
(-b^{-1},\ast,\ast,\ast)
&\longrightarrow (\ast,\ast,0,-b^{-1})
\longrightarrow (\infty,0,\infty,\ast)
\longrightarrow (0,\infty,\ast,\infty)\\
&\longrightarrow (\ast,\ast,-b^{-1},0)
\longrightarrow (\ast,-b^{-1},\ast,\ast),
\end{aligned}
\end{equation}
and
\begin{equation}
\label{eq:4d-conf-x3}
\begin{aligned}
(\ast,\ast,-b^{-1},\ast)
&\longrightarrow (0,-b^{-1},\ast,\ast)
\longrightarrow (\infty,\ast,\infty,0)
\longrightarrow (\ast,\infty,0,\infty)\\
&\longrightarrow (-b^{-1},0,\ast,\ast)
\longrightarrow (\ast,\ast,\ast,-b^{-1}).
\end{aligned}
\end{equation}
Here \(\ast\) denotes a generic finite nonzero value unless otherwise specified.
The singularity pattern starting from \(x_1=0\) is the following cyclic pattern of period \(14\):
\begin{equation}
\label{eq:4d-cyclic-x1zero}
\begin{aligned}
(0,\ast,\ast,\ast)
&\to (\infty,\ast,\infty,0)
\to (\ast,\infty,0,\infty)
\to (\ast,0,\ast,\ast)\\
&\to (\ast,\ast,\infty,\ast)
\to (\ast,\infty,0,\ast)
\to (\infty,0,\ast,\ast)
\to (0,\ast,\infty,\infty)\\
&\to (\ast,\infty,\ast,0)
\to (\infty,\ast,0,\ast)
\to (\ast,0,\infty,\infty)
\to (0,\infty,\ast,\ast)\\
&\to (\infty,\ast,\ast,0)
\to (\ast,\ast,\ast,\infty)
\to (0,\ast,\ast,\ast).
\end{aligned}
\end{equation}
Similarly, the singularity pattern starting from \(x_3=0\) is
\begin{equation}
\label{eq:4d-cyclic-x3zero}
\begin{aligned}
(\ast,\ast,0,\ast)
&\to (\infty,0,\infty,\ast)
\to (0,\infty,\ast,\infty)
\to (\ast,\ast,\ast,0)\\
&\to (\infty,\ast,\ast,\ast)
\to (0,\ast,\ast,\infty)
\to (\ast,\ast,\infty,0)
\to (\infty,\infty,0,\ast)\\
&\to (\ast,0,\ast,\infty)
\to (0,\ast,\infty,\ast)
\to (\infty,\infty,\ast,0)
\to (\ast,\ast,0,\infty)\\
&\to (\ast,0,\infty,\ast)
\to (\ast,\infty,\ast,\ast)
\to (\ast,\ast,0,\ast).
\end{aligned}
\end{equation}

Define the multidegrees corresponding to the generic parts of the singular divisors by
\[
\begin{aligned}
\mathbf t_n
&:= \operatorname{mdeg}\bigl((x_{1,n}=-b^{-1})^\circ\bigr),\\
\mathbf u_n
&:= \operatorname{mdeg}\bigl((x_{3,n}=-b^{-1})^\circ\bigr),\\
\mathbf f_n
&:= \operatorname{mdeg}\bigl((x_{1,n}=0)^\circ\bigr),\\
\mathbf g_n
&:= \operatorname{mdeg}\bigl((x_{3,n}=0)^\circ\bigr).
\end{aligned}
\]
Then the singularity patterns \eqref{eq:4d-conf-x1} and \eqref{eq:4d-conf-x3}, together with Halburd-type divisor decompositions, give
\begin{equation}
\label{eq:4d-confined-degree}
\mathbf d_n=\mathbf t_n+\mathbf u_{n-4},
\qquad
\mathbf e_n=\mathbf u_n+\mathbf t_{n-4}.
\end{equation}

Similarly, the cyclic patterns \eqref{eq:4d-cyclic-x1zero} and \eqref{eq:4d-cyclic-x3zero} give
\[
\mathbf d_n
=
\mathbf t_{n-3}
+\mathbf u_{n-1}
+\mathbf f_n
+\mathbf f_{n-7}
+\mathbf f_{n-11}
+\mathbf g_{n-2}
+\mathbf g_{n-5}
+\mathbf g_{n-9},
\]
\[
\mathbf e_n
=
\mathbf u_{n-3}
+\mathbf t_{n-1}
+\mathbf g_n
+\mathbf g_{n-7}
+\mathbf g_{n-11}
+\mathbf f_{n-2}
+\mathbf f_{n-5}
+\mathbf f_{n-9}.
\]

Therefore
\begin{equation}
\label{eq:4d-cyclic-tu}
\begin{aligned}
\mathbf t_n
=&
\mathbf t_{n-3}
+\mathbf u_{n-1}
-\mathbf u_{n-4}
+\mathbf f_n
+\mathbf f_{n-7}
+\mathbf f_{n-11}
+\mathbf g_{n-2}
+\mathbf g_{n-5}
+\mathbf g_{n-9},\\
\mathbf u_n
=&
\mathbf u_{n-3}
+\mathbf t_{n-1}
-\mathbf t_{n-4}
+\mathbf g_n
+\mathbf g_{n-7}
+\mathbf g_{n-11}
+\mathbf f_{n-2}
+\mathbf f_{n-5}
+\mathbf f_{n-9}.
\end{aligned}
\end{equation}

Set
\[
\boldsymbol\epsilon_1=(1,0,0,0),\quad
\boldsymbol\epsilon_2=(0,1,0,0),\quad
\boldsymbol\epsilon_3=(0,0,1,0),\quad
\boldsymbol\epsilon_4=(0,0,0,1).
\]
From the cyclic pattern of period \(14\), the sequences \(\mathbf f_n\) and \(\mathbf g_n\) are periodic with period \(14\), and
\begin{equation}
\label{eq:4d-fg-period}
\mathbf f_n=
\begin{cases}
\boldsymbol\epsilon_1, & n\equiv 0 \pmod {14},\\
\boldsymbol\epsilon_4, & n\equiv 1 \pmod {14},\\
\boldsymbol\epsilon_3, & n\equiv 10 \pmod {14},\\
\boldsymbol\epsilon_2, & n\equiv 11 \pmod {14},\\
\mathbf 0, & \text{otherwise},
\end{cases}
\qquad
\mathbf g_n=
\begin{cases}
\boldsymbol\epsilon_3, & n\equiv 0 \pmod {14},\\
\boldsymbol\epsilon_2, & n\equiv 1 \pmod {14},\\
\boldsymbol\epsilon_1, & n\equiv 10 \pmod {14},\\
\boldsymbol\epsilon_4, & n\equiv 11 \pmod {14},\\
\mathbf 0, & \text{otherwise}.
\end{cases}
\end{equation}
Here the indices of \(\mathbf f_n\) and \(\mathbf g_n\) are interpreted modulo \(14\).

The initial values are obtained by direct computation:
\begin{equation}
\label{eq:4d-tu-de-initial}
\begin{aligned}
\mathbf t_0=\mathbf d_0&=(1,0,0,0),&
\mathbf t_1=\mathbf d_1&=(1,0,1,1),&
\mathbf t_2=\mathbf d_2&=(1,1,2,1),&
\mathbf t_3=\mathbf d_3&=(3,2,2,1),\\
\mathbf u_0=\mathbf e_0&=(0,0,1,0),&
\mathbf u_1=\mathbf e_1&=(1,1,1,0),&
\mathbf u_2=\mathbf e_2&=(2,1,1,1),&
\mathbf u_3=\mathbf e_3&=(2,1,3,2).
\end{aligned}
\end{equation}

Thus the initial values \eqref{eq:4d-tu-de-initial}, the periodic sequences \eqref{eq:4d-fg-period}, and the recurrence \eqref{eq:4d-cyclic-tu} determine \(\mathbf t_n\) and \(\mathbf u_n\) successively for \(n\geq4\). Finally, \eqref{eq:4d-confined-degree} gives \(\mathbf d_n\) and \(\mathbf e_n\) for \(n\geq4\). For reference, the first values obtained in this way are 
\[ \begin{array}{c|c|c} n & \mathbf d_n & \mathbf e_n\\ \hline 0 & (1,0,0,0) & (0,0,1,0)\\ 1 & (1,0,1,1) & (1,1,1,0)\\ 2 & (1,1,2,1) & (2,1,1,1)\\ 3 & (3,2,2,1) & (2,1,3,2)\\ 4 & (3,2,4,3) & (4,3,3,2)\\ 5 & (5,4,6,3) & (6,3,5,4)\\ 6 & (8,6,6,5) & (6,5,8,6)\\ 7 & (9,6,10,8) & (10,8,9,6)\\ 8 & (11,10,12,9) & (12,9,11,10)\\ 9 & (15,12,14,11) & (14,11,15,12)\\ 10 & (17,14,18,15) & (18,15,17,14). \end{array} 
\] 
If \(\delta_n:=\deg \varphi^n\) denotes the scalar degree, obtained by taking the maximum among the components of the multidegrees of \(x_{1,n},x_{2,n},x_{3,n},x_{4,n}\), then \[ \delta_0,\delta_1,\delta_2,\ldots = 1,1,2,3,4,6,8,10,12,15,18,21,25,28,33,\ldots . \] Thus the recurrence system above gives an explicit degree sequence for the four-dimensional map. In particular, the sequence exhibits quadratic degree growth, and hence has vanishing algebraic entropy \(\mathcal E=\lim_{n\to\infty}n^{-1}\log\delta_n\).

\section{Degree-drop divisors for a Riccati-type linearizable map}
\label{sec:riccati-valuation-degree-drop}

In this section we treat an example of a linearizable map due to Takenawa--Eguchi--Grammaticos--Ohta--Ramani--Satsuma \cite{TakenawaEguchiGrammaticosOhtaRamaniSatsuma2003}. That paper treats the non-autonomous case, whereas here we restrict ourselves to an autonomous map and analyze it from the viewpoint of divisor sequences on the orbit graph and degree-drop relations.

Let the space at time \(n\) be \(X_n=\PP^1\times\PP^1\), with affine coordinates \((x_n,y_n)\). The map considered in this paper is
\begin{equation}
\label{eq:riccati-map-autonomous}
\varphi:(x_n,y_n)\mapsto (x_{n+1}, y_{n+1}), \qquad
x_{n+1}=y_n,
\qquad
y_{n+1}
=
\frac{y_n(a x_n-y_n)}{x_n},
\end{equation}
where \(a\in \mathbb C^*\) is a constant.

\begin{figure}[htbp]
\centering
\includegraphics[width=.4\textwidth]{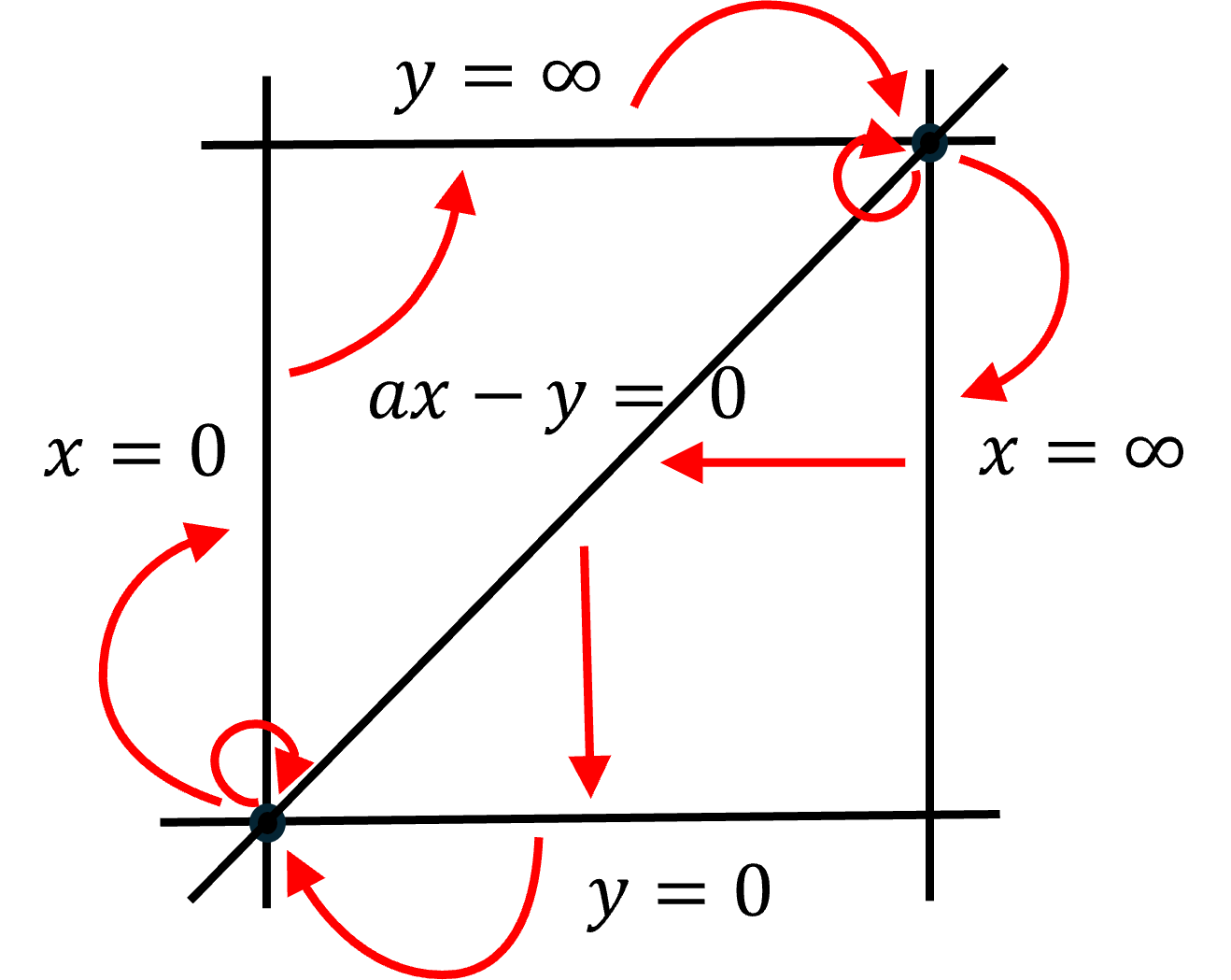}
\caption{A linearizable map.}
\label{fig:Ex8}
\end{figure}

The singularity patterns are
\begin{equation}
\label{eq:riccati-singular-sequence}
\cdots \to (\infty,\infty)
\to
(x=\infty)^\circ
\to
(a x -y)^\circ
\to
(y=0)^\circ
\to
(0,0)
\to
(0,0)
\to
\cdots
\end{equation}
and
\begin{equation}
\label{eq:riccati-singular-sequence2}
\cdots \to
(0,0)
\to
(0,0)
\to
(x=0)^\circ
\to
(y=\infty)^\circ
\to
(\infty, \infty)
\to
(\infty,\infty)
\to
\cdots
\end{equation}
(Figure~\ref{fig:Ex8}).

For example, at time \(j\), put
\[
(x_j,y_j)=(c,\varepsilon)
\]
along \(y=0\), where \(c\neq0\) is a generic constant and \(\varepsilon\) is a small parameter. Then, for \(r\geq0\),
\begin{equation}
\label{eq:riccati-orders-xy}
\ord_\varepsilon x_{j+r}
=
\left\lceil\frac r2\right\rceil,
\qquad
\ord_\varepsilon y_{j+r}
=
\left\lfloor\frac r2\right\rfloor+1.
\end{equation}
Similarly, putting $(x_j,y_j)=(c,\varepsilon^{-1})$, with $c\ne0$
generic, gives
\begin{equation}
\label{eq:riccati-orders-xy2}
\ord_\varepsilon x_{j+r}=-r,
\qquad
\ord_\varepsilon y_{j+r}=-(r+1)
\qquad (r\geq0).
\end{equation}

Next we compute the degree-drop divisor for
\[
D_{n+1}:=(y_{n+1}=0).
\]
The line bundle $\OO_{X_{n+1}}(D_{n+1})\simeq\OO(0,1)$ is
globally generated. Choose the basis $s_0=1$, $s_1=1/y_{n+1}$
of its global sections. By \eqref{eq:riccati-map-autonomous},
\[
\varphi^*D_{n+1}=(y_n=0)+(a x_n-y_n=0).
\]
Near $(x_n,y_n)=(0,0)$, this divisor has local equation
$b_0=y_n(a x_n-y_n)$. The corresponding regular functions
$a_{i,0}=b_0(s_i\circ\varphi)$ are
\[
a_{0,0}=y_n(a x_n-y_n),\qquad a_{1,0}=x_n.
\]

Near $(x_n,y_n)=(\infty,\infty)$, put
$u=1/x_n$ and $v=1/y_n$. Since
\[
y_{n+1}=\frac{av-u}{v^2},
\]
the local equation of $\varphi^*D_{n+1}$ is $b_\infty=av-u$.
The corresponding regular functions are
\[
a_{0,\infty}=av-u,\qquad a_{1,\infty}=v^2.
\]
Both pairs have a common zero at the indicated point, and the
remaining affine charts have no additional common zeros. Hence
\[
I_{D_{n+1}}(\varphi)=\{(0,0)_n,(\infty,\infty)_n\}.
\]
These are exactly the simultaneous zeros of the primitive numerator
and denominator pairs for the $y$-coordinate, as in
\eqref{eq:coordinate-pair-indeterminacy}.  In particular,
$(\infty,\infty)$ is a genuine base point of that coordinate map,
not an additional point introduced by the definition via $|D|$;
it is invisible if one checks only the finite affine chart.

For $n\geq1$, write $K_{n+1}^{(0)}(D_{n+1})$ and
$K_{n+1}^{(\infty)}(D_{n+1})$ for the contributions from the
two points, respectively. By \eqref{eq:iterate-defect-divisor2},
\begin{equation}
\label{eq:riccati-degree-lower}
\begin{aligned}
K_{n+1}(D_{n+1})
&=(\varphi^n)^*\varphi^*D_{n+1}
 -(\varphi^{n+1})^*D_{n+1}\\
&=K_{n+1}^{(0)}(D_{n+1})
 +K_{n+1}^{(\infty)}(D_{n+1}).
\end{aligned}
\end{equation}

The possible prime divisors on $X_0$ contributing to the first
term are those in the first singularity pattern
\eqref{eq:riccati-singular-sequence}, namely
\[
(y_0=0)^\circ,\qquad (y_1=0)^\circ,\qquad (y_2=0)^\circ.
\]
At time $0$, these correspond to $(y_0=0)$,
$(a x_0-y_0=0)$, and $(x_0=\infty)$, respectively.
Using the local functions $a_{i,0}$ and the orders
\eqref{eq:riccati-orders-xy}, we find
\[
\begin{aligned}
\mu_{n,(y_j=0)^\circ}(D_{n+1})
&=\min\bigl\{\ord_{(y_j=0)^\circ}\bigl(y_n(a x_n-y_n)\bigr),
              \ord_{(y_j=0)^\circ}(x_n)\bigr\}\\
&=\left\lceil\frac{n-j}{2}\right\rceil,
\qquad j=0,1,2,
\end{aligned}
\]
where coefficients of components not yet entering $(0,0)_n$
are understood to be zero. Consequently,
\begin{equation}
\label{eq:riccati-K-factorization}
\begin{aligned}
K_{n+1}^{(0)}(D_{n+1})={}&
 \left\lceil\frac n2\right\rceil(y_0=0)^\circ
 +\left\lceil\frac{n-1}{2}\right\rceil(y_1=0)^\circ
%\\
%&
+\left\lceil\frac{n-2}{2}\right\rceil(y_2=0)^\circ.
\end{aligned}
\end{equation}

The second singularity pattern \eqref{eq:riccati-singular-sequence2}
shows that the prime divisors contributing to
$K_{n+1}^{(\infty)}(D_{n+1})$ are $(x_0=0)^\circ$ and
$(y_0=\infty)^\circ$. For completeness, the iterates satisfy
\[
y_n=\frac{y_0^{\lfloor n/2\rfloor+1}
             (a x_0-y_0)^{\lceil n/2\rceil}}{x_0^n},
\qquad x_n=y_{n-1}\quad (n\geq1),
\]
as follows from $y_{k+1}/x_{k+1}=a-y_k/x_k$.
This expression also confirms that the two singularity patterns
account for all prime divisors on $X_0$ that can contribute to
\eqref{eq:riccati-degree-lower}.

Along $(y_0=\infty)^\circ$, the orders of
$u=1/x_n$ and $v=1/y_n$ are $n$ and $n+1$,
respectively. Thus the local functions
$a_{0,\infty}=av-u$ and $a_{1,\infty}=v^2$ give
\[
\mu_{n,(y_0=\infty)^\circ}(D_{n+1})=n.
\]
Similarly, along $(x_0=0)^\circ$, the orders of
$u$ and $v$ are $n-1$ and $n$ for $n\geq2$, giving
\[
\mu_{n,(x_0=0)^\circ}(D_{n+1})=n-1.
\]
For $n=1$ this latter coefficient is zero, since the divisor
has not yet entered $(\infty,\infty)$. Therefore
\begin{equation}
\label{eq:riccati-K-infinity}
K_{n+1}^{(\infty)}(D_{n+1})
=(n-1)(x_0=0)^\circ+n(y_0=\infty)^\circ
\qquad(n\geq1).
\end{equation}

We now compute the bidegrees on $X_0$. For the three divisors
in \eqref{eq:riccati-K-factorization} we have
\begin{equation}
\label{eq:riccati-degree-yj-zero}
\operatorname{bideg}(y_0=0)^\circ=(0,1),\quad
\operatorname{bideg}(y_1=0)^\circ=(1,1),\quad
\operatorname{bideg}(y_2=0)^\circ=(1,0).
\end{equation}
The two divisors in \eqref{eq:riccati-K-infinity} have bidegrees
$\operatorname{bideg}(x_0=0)^\circ=(1,0)$ and
$\operatorname{bideg}(y_0=\infty)^\circ=(0,1)$.
It follows that
\[
\operatorname{bideg}K_{n+1}^{(0)}(D_{n+1})=(n-1,n),
\qquad
\operatorname{bideg}K_{n+1}^{(\infty)}(D_{n+1})=(n-1,n),
\]
and hence
\begin{equation}
\label{eq:riccati-K-total-bideg}
\operatorname{bideg}K_{n+1}(D_{n+1})
= (2n-2,2n)\qquad(n\geq1).
\end{equation}

Since $(a x-y=0)$ is linearly equivalent to
$(x=0)+(y=0)$ on $\PP^1\times\PP^1$, the additivity of
Cartier pullbacks gives
\[
\operatorname{bideg}\bigl((\varphi^n)^*\varphi^*D_{n+1}\bigr)
=\operatorname{bideg}x_n+2\operatorname{bideg}y_n.
\]
Taking bidegrees in \eqref{eq:riccati-degree-lower}, we obtain
\[
(2n-2,2n)
=\operatorname{bideg}x_n+2\operatorname{bideg}y_n
 -\operatorname{bideg}y_{n+1}.
\]
Together with $\operatorname{bideg}x_n=\operatorname{bideg}y_{n-1}$
for $n\geq1$ and the initial values
\[
\operatorname{bideg}y_0=(0,1),\qquad
\operatorname{bideg}y_1=(1,2),
\]
this gives, by induction,
\begin{equation}
\label{eq:riccati-k-general}
\operatorname{bideg}y_n=(n,n+1)\quad(n\geq0),
\qquad
\operatorname{bideg}x_n=(n-1,n)\quad(n\geq1).
\end{equation}

\begin{remark}
In fact, by blowing up \((0,0)\) and \((\infty,\infty)\), the degree-drop divisors in the iteration disappear and the map \(\varphi\) becomes algebraically stable. If the blown-up space is denoted by \(\widetilde X\), then
\[
(\varphi^n)^*=(\varphi^*)^n
\]
holds on \(\operatorname{Pic}(\widetilde X)\), and the computation reduces to a finite-dimensional one. However, incorporating the exceptional divisors themselves, added by the blow-ups of \((0,0)\) and \((\infty,\infty)\), as divisor sequences on the present orbit graph is not part of the framework of this paper.
\end{remark}

\section{Concluding remarks}

In this paper we treated two kinds of degree relations in the
divisor-orbit decomposition method.  The first kind is obtained from
Halburd-type divisor decompositions of time-indexed divisors on the
normalized orbit graph, as in \eqref{eq:degree-from-halburd-decomposition}.
For divisors with globally generated associated line bundles, the second
kind is obtained from effective degree-drop divisors, or equivalently
from the positive-coefficient components of the functoriality defect of
pullbacks, as in \eqref{eq:degree-drop-relation}.  On projective and
multiprojective spaces these components are computed directly from
homogeneous coordinate tuples and primitive coordinate pairs.  The
intrinsic formulation detects them when, during the iteration, a prime
divisor is mapped into the divisor-dependent indeterminacy locus of the
chosen divisor.

These two kinds of relations need not involve the same divisorial
components.  A Halburd-type divisor decomposition depends on the chosen
time-indexed divisor and may contain extra components that do not directly
correspond to degree drops.  
This distinction is visible in Example~\ref{example:p1p1-degree-drop-orbit},
where a degree-drop relation is nontrivial even though
the corresponding Halburd-type divisor decomposition
is trivial. Thus, to apply the divisor-orbit decomposition
method, one has to identify the divisorial components appearing in both
kinds of relations.

If these components can be described, up to time shifts, by finitely many
divisor sequences, then Proposition~\ref{prop:finite-type-closed-linear-system}
gives a closed system of finitely many linear difference relations for the
degree sequence and the auxiliary degree sequences.  On the other hand,
the Riccati-type linearizable example in
Section~\ref{sec:riccati-valuation-degree-drop} shows that, if one stays
on the original space, the divisor sequences on the orbit graph need not
close into a finite set.  Even in such cases, degree relations can
sometimes be derived by explicitly tracking the singularity patterns.

As the example in Section~\ref{sec:riccati-valuation-degree-drop}
suggests, finite-dimensional relations may sometimes become visible after
passing to divisor classes or after choosing another birational model.  It
would be useful to clarify how such reductions are related to the divisor
computations on orbit graphs developed in this paper.

Another natural direction is to pass from finite-window orbit graphs to
the projective limit over all finite windows.  On such a full orbit
object, the time shift should act as a natural automorphism, rather than
as a correspondence between different finite windows.  This viewpoint may
be useful for studying divisor classes along the orbit, and possibly for
questions related to conserved quantities.

\section*{Statements and declarations}

\noindent\textbf{Funding.}
This work was supported by the Japan Society for the Promotion of Science
(JSPS) KAKENHI Grant Numbers JP22K03383 and JP26K01144.

\medskip
\noindent\textbf{Competing interests.}
The author has no relevant financial or non-financial interests to disclose.

\medskip
\noindent\textbf{Data availability.}
No datasets were generated or analysed during the current study.

\medskip
\noindent\textbf{Use of artificial intelligence.}
During the preparation of this manuscript, the author used ChatGPT (OpenAI) for language editing and auxiliary checks of mathematical calculations and notation. The correctness and logical consistency of all proofs and mathematical results were independently verified by the author, who takes full responsibility for the content of the manuscript.

\end{document}